\documentclass{amsart}
\usepackage{latexsym,amsxtra,amscd,ifthen}
\usepackage{amsfonts}
\usepackage{verbatim}
\usepackage{amsmath}
\usepackage{amsthm}
\usepackage{amssymb}
\numberwithin{equation}{section}

\theoremstyle{plain}
\newtheorem{theorem}{Theorem}[section]
\newtheorem{lemma}[theorem]{Lemma}
\newtheorem{proposition}[theorem]{Proposition}
\newtheorem{corollary}[theorem]{Corollary}

\theoremstyle{definition}
\newtheorem{definition}[theorem]{Definition}
\newtheorem{example}[theorem]{Example}

\newtheorem{remark}[theorem]{Remark}
\newtheorem{question}[theorem]{Question}

\makeatletter              % This sequence of commands will
\let\c@equation\c@theorem  % incorporate equation numbering
\makeatother

\newcommand{\bfl}{\mathfrak l}
\newcommand{\lrr}{l_{\Re}}

\DeclareMathOperator{\hdet}{hdet} 
\DeclareMathOperator{\gldim}{gldim}

\DeclareMathOperator{\Ext}{Ext}

\DeclareMathOperator{\Aut}{Aut}

\DeclareMathOperator{\injdim}{injdim}

\DeclareMathOperator{\GKdim}{GKdim}
 \DeclareMathOperator{\R}{R}
 \DeclareMathOperator{\Hom}{Hom}

\DeclareMathOperator{\GrMod}{{\sf GrMod}}

\newcommand{\fm}{\mathfrak{m}}
\newcommand{\fp}{\mathfrak{p}}

\newcommand{\be}{\begin{enumerate}}
\newcommand{\ee}{\end{enumerate}}
\newcommand{\bq}{\begin{eqnarray*}}
\newcommand{\eq}{\end{eqnarray*}}
\newcommand{\bqn}{\begin{eqnarray}}
\newcommand{\eqn}{\end{eqnarray}}

\begin{document}
\title[Dual Reflection Groups]
{Nakayama automorphism and rigidity of\\
Dual Reflection Group coactions}

\author{E. Kirkman, J. Kuzmanovich and J.J. Zhang}

\address{Kirkman: Department of Mathematics,
P. O. Box 7388, Wake Forest University, Winston-Salem, NC 27109}

\email{kirkman@wfu.edu}
%\thanks{kirkman partially supported by Simons Foundation}

\address{Kuzmanovich: Department of Mathematics,
P. O. Box 7388, Wake Forest University, Winston-Salem, NC 27109}

\email{kuz@wfu.edu}

\address{Zhang: Department of Mathematics, Box 354350,
University of Washington, Seattle, Washington 98195, USA}

\email{zhang@math.washington.edu}
%\thanks{zhang partially supported by NSF}

\begin{abstract}
We study homological properties and rigidity of
group coactions on Artin-Schelter regular algebras.
\end{abstract}


% 16E10  View Publications  (1991-now) Homological dimension 

%16W30 (1991-now) Coalgebras, bialgebras,
%      Hopf algebras [See also 16S40, 57T05];

%20J05 (1973-now) Homological methods in group theory


%\thanks{the US National Science Foundation.}

\maketitle

%\tableofcontents

\setcounter{section}{-1}

\section{Introduction}
\label{xxsec0}

The classical Shephard-Todd-Chevalley Theorem states that if $G$ is 
a finite group acting faithfully on a finite dimensional 
${\mathbb C}$-vector space $\bigoplus_{i=1}^n {\mathbb C} x_i$, 
then the fixed subring ${\mathbb C}[x_1,\cdots, x_n]^G$ is isomorphic 
to ${\mathbb C}[x_1,\cdots, x_n]$ if and only if $G$ is generated by 
pseudo-reflections of $\bigoplus_{i=1}^n {\mathbb C} x_i$. 
Such a group $G$ is called a {\it reflection group}. 
The indecomposable complex reflection groups are classified 
by Shephard and Todd \cite{ST}, and there are three infinite families 
plus 34 exceptional groups. When the commutative polynomial rings are 
replaced by skew polynomial rings, one can define a notion of reflection group 
in this noncommutative setting. Then there is one extra class of ``mystic'' 
reflection groups $M(a,b,c)$, discovered in \cite{KKZ2}, also see \cite{BB} 
for further discussion.

When $G$ is replaced by a semisimple Hopf algebra $H$ (which is not a 
group algebra), then there is no inner faithful action of $H$ on the 
commutative polynomial ring by a very 
nice result of Etingof-Walton \cite{EW1}. This is one of many reasons
why we need to consider noncommutative Artin-Schelter 
regular algebras if we want to have a fruitful noncommutative invariant 
theory. Artin-Schelter regular algebras were introduced by Artin-Schelter 
\cite{AS} in 1980's, and by now, are considered  a natural analogue of 
commutative polynomial rings in many respects. The definition of an Artin-Schelter 
regular algebra (abbreviated by AS regular) is given in Definition \ref{xxdef1.1}. 

Throughout the rest of this paper, let $\Bbbk$ be a base field of characteristic 
zero, and all vector spaces, (co)algebras, and morphisms are over $\Bbbk$.
Let $H$ be a semisimple Hopf algebra and let $K$ be the 
$\Bbbk$-linear dual of $H$. Then $K$ is also a semisimple 
Hopf algebra. It is well-known that a left $H$-action
on an algebra $A$ is equivalent to a right $K$-coaction 
on $A$, and we will use this fact freely.

In this paper we are interested in the case
when $H$ is $\Bbbk^G:=\Hom_{\Bbbk}(\Bbbk G, \Bbbk)$, or 
equivalently, $K$ is the group algebra $\Bbbk G$ for some finite
group $G$. Let $e$ denote the unit of $G$. As an algebra,
$\Bbbk^G=\bigoplus_{g\in G} \Bbbk \; p_g$ where its multiplication 
is determined by
$$p_g p_h=\begin{cases} p_g & g=h,\\0 & g\neq h \end{cases}, 
\quad 1=\sum_{g\in G} p_g,$$
and its coalgebra structure is determined by
$$\Delta(p_g)=\sum_{h\in G} p_h \otimes p_{h^{-1}g}, \quad 
\epsilon(p_{g})=\delta_{e,g}:=\begin{cases} 1 & g=e,\\ 0 & g\neq e,
\end{cases}$$
for all $g\in G$.

\begin{definition}
\label{xxdef0.1} A finite group $G$ is called a 
{\it dual reflection group} if the Hopf algebra $H:=\Bbbk^G$ acts
homogeneously and inner faithfully on a noetherian AS regular 
domain $A$ generated in degree 1 such that the fixed subring 
$A^H$ is again AS regular, i.e. the identity component of $A$ under the $G$-grading is AS regular. In this case we say that 
{\it $G$ coacts on $A$ as a dual reflection group}.  
\end{definition}

An example of a dual reflection group acting on an
AS regular algebras is given in Example \ref{xxex3.7}; further examples will appear in
\cite{KKZ5}.  We do not yet have a classification of dual reflection groups.

When we have a Hopf algebra (co)action on an AS regular
algebra, the homological (co)determinant is defined 
in \cite{KKZ3}.

\begin{definition}
\label{xxdef0.2} Suppose a finite group $G$ coacts on $A$ as a
dual reflection group. Let $D$ be the homological codeterminant
of the $\Bbbk G$-coaction on $A$ as defined in 
\cite[Definition {\rm{6.2}}]{KKZ3} or \cite[Definition 
{\rm{1.4(b)}}]{CWZ}. Then we call the element $m:=D^{-1}$ in $G$ the 
{\it mass} element of the $\Bbbk G$-coaction (or $\Bbbk^G$-action) 
on $A$.
\end{definition}

Our first result is the following. Let $H_A(t)$
be the Hilbert series of a graded algebra $A$. 

\begin{theorem}
\label{xxthm0.3} Let $A$ be a noetherian AS regular domain 
generated in degree 1. Let $G$ coact on $A$ inner faithfully 
as a dual reflection group and $H=\Bbbk^G$. Then the following hold.
\begin{enumerate}
\item[(1)]
There is a set of homogeneous elements $\{ f_{g}\mid g\in G\}\subseteq A$ 
with $f_e=1$ such that $A=\bigoplus_{g\in G} A_g$ and 
$A_g=f_g \cdot A^H=A^H \cdot f_g$ for all $g\in G$. 
\item[(2)]
There is a generating subset $\Re$ of the group $G$ satisfying
$e\not\in \Re$ such that $A_1=\oplus_{g\in \Re} \Bbbk f_g \oplus (A_1\cap A^H)$.
\item[(3)]
Let $\lrr(g)$ denote the (reduced) length of $g\in G$ with respect to 
$\Re$. Then $\deg f_g=\lrr(g)$ for all $g\in G$.
\item[(4)]
The mass element $m$ is the unique element in $G$ of the maximal length with 
respect to $\lrr$ defined as in part {\rm{(3)}}.
\item[(5)]
Let $\fp(t)=H_{A}(t) H_{A^H}(t)^{-1}$. Then $\fp(t)$ is a product 
of cyclotomic polynomials, $\fp(1)=|G|$ and $\deg \fp(t)=\lrr(m)$.
\end{enumerate}
\end{theorem}

By part (4) of the above theorem, $m$ is also the {\it malth} (= maximal length) 
element in $G$. Hence a purely homologically defined invariant $m$ has strong 
combinatorial flavor. Indeed this mass element will appear in several 
other results.

The Nakayama automorphism of an AS regular (or AS Gorenstein) algebra is an 
important invariant of the algebra. In the study of noetherian 
Hopf algebras, the explicit expression of the Nakayama automorphism has several 
applications in Poincar{\'e} duality \cite[Corollary 0.4]{BZ}, Radford's $S^4$ 
formula \cite[Theorem 0.6]{BZ}, and so on. The result about the Nakayama 
automorphism of a smash product \cite[Theorem 0.2]{RRZ2} partially recovers  a 
number of previous results concerning Calabi-Yau algebras; see the discussion in 
\cite[Introduction]{RRZ2}. 
In \cite{RRZ3}, a further connection between the Nakayama automorphism 
of an AS regular algebra and the Nakayama automorphism of  its Ext-algebra
is studied. In \cite{LMZ}, the Nakayama automorphism 
has been used to study the automorphism group and  a cancellation problem
of algebras.  The Nakayama 
automorphism is also essential in the study of 
rigid dualizing complexes \cite{BZ, VdB}. The main purpose of this paper is to 
describe the Nakayama automorphism and to study the homological properties of 
various algebras related to the dual reflection group coaction on AS regular algebras.

We will recall several definitions in Section \ref{xxsec1}.
Let $A$ be any algebra. We use $\mu_A$ to denote the Nakayama 
automorphism of $A$. If $A$ is a connected graded algebra,
then the graded Nakayama automorphism $\mu_A$ is unique (if exists) since 
there is no nontrivial invertible homogeneous element in $A$. When 
$A$ is a connected graded noetherian AS Gorenstein algebra, we usually 
require $\mu_A$ to be the graded Nakayama automorphism of $A$. 

From now on, let $(f_g, \Re, m, \fp(t))$ be as in Theorem \ref{xxthm0.3}.
In the setting of Theorem \ref{xxthm0.3}, we define the {\it covariant ring} of the
$H$-action on $A$ to be 
$$A^{cov \; H}:=A/((A^H)_{\geq 1})$$
where $((A^H)_{\geq 1})$ is the ideal of $A$ generated by homogeneous
invariant elements in $A^H$ of positive degree. The covariant ring 
is also called the coinvariant ring by some authors. 

\begin{theorem}
\label{xxthm0.4} 
Under the hypothesis of Theorem {\rm{\ref{xxthm0.3}}}.
\begin{enumerate}
\item[(1)]
The covariant ring $A^{cov \; H}$ is Frobenius of 
dimension $|G|$. 
\item[(2)]
The Hilbert series of $A^{cov\; H}$ is $\fp(t)$. As a consequence,
the coefficients of $\fp(t)$ are positive and the coefficient of the
leading term of $\fp(t)$ is 1. 
\item[(3)]
Let $\overline{f}_g$ be the homomorphic image of $f_g$ in
$A^{cov\; H}$ for each $g\in G$. Then $\{\overline{f}_g \mid g\in G\}$
is a $\Bbbk$-linear basis of $A^{cov\; H}$ and the graded 
Nakayama automorphism of $A^{cov \; H}$ is of the form
$$\mu_{A^{cov\; H}}: \overline{f}_g\to \beta(g) \overline{f}_{mgm^{-1}}, 
\quad {\text{for all $g\in G$}},$$ 
where $\{\beta(g)\}_{g\in G}$ are some nonzero scalars in $\Bbbk$.
\item[(4)]
The scalars in part {\rm{(3)}} satisfy
$$\beta(m)=1,$$ and
if $m$ commutes with $g$ and $h$ and $\lrr(gh)=\lrr(g)+\lrr(h)$ then
$$\beta(gh)=\beta(g)\beta(h).$$
\end{enumerate}
\end{theorem}

As in Theorems \ref{xxthm0.3} and \ref{xxthm0.4}, the mass element $m$ 
has special properties. We can say more next.

\begin{theorem}
\label{xxthm0.5} Under the hypothesis of Theorem {\rm{\ref{xxthm0.3}}}.
\begin{enumerate}
\item[(1)]
$f_m$ is a homogeneous normal element in $A$.  
\item[(2)]
Let $\phi_m$ be the conjugation 
automorphism of $A$ defined by
$$a\to f_{m} a f_{m}^{-1}, \quad {\text{for all $a\in A$.}}$$
Then 
$$\phi_m: f_g\to \beta(g) f_{mgm^{-1}}, \quad {\text{for all $g\in G$}},$$
where $\beta(g)$ are defined as in Theorem {\rm{\ref{xxthm0.4}(3)}}.
In other words, the graded Nakayama automorphism of the covariant ring
$A^{cov\; H}$ is induced by the conjugation automorphism $\phi_m$.
\item[(3)]
Any ${\mathbb N}$-graded algebra automorphism of $A^{cov\; H}$ 
commutes with $\mu_{A^{cov\; H}}$. 
\item[(4)]
The automorphism $\phi_m$ sends $A^H$ to $A^H$. 
\end{enumerate}
\end{theorem}

The next result tells us about the Nakayama automorphisms
of other algebras related to $A$. The homological determinant of a Hopf
algebra action is defined in \cite[Definition 3.3]{KKZ3}.

\begin{theorem}
\label{xxthm0.6} 
Under the hypothesis of Theorem {\rm{\ref{xxthm0.3}}}. 
Let $\eta_m=\phi_{m}^{-1}$. 
\begin{enumerate}
\item[(1)]
$\mu_{A^H}=(\eta_{m} \circ \mu_{A})\mid_{A^H}$.
As a consequence, $\mu_A$ sends $A^H$ to $A^H$.
\item[(2)]
Three automorphisms $\mu_{A^H}$, $\mu_{A}\mid_{A^H}$ and $\eta_{m}\mid_{A^H}$
commute with each other.
\item[(3)]
The homological determinant of the $H$-action on $A$
is the projection from $H\to \Bbbk \; p_{m^{-1}}$. 
\item[(4)]
Let $trans^l_{m}$ be the left translation automorphism 
of $H=\Bbbk^G$ defined by 
$$trans^l_{m}: p_g \to p_{mg},
\quad {\text{for all $g\in G$.}}$$ 
Then the Nakayama automorphism
of $A\# H$ is equal to $\mu_A \# trans^l_{m}$.
\end{enumerate}
\end{theorem}
 
Classically, when $A$ is a commutative polynomial ring and $G$ is a reflection
group, the covariant ring is clearly a complete intersection,
so the following is a natural question.  
\begin{question}
\label{xxque0.7}
Under the hypothesis of Theorem \ref{xxthm0.3} (or more generally for a semisimple Hopf algebra $H$), is 
the covariant ring 
$A^{cov \; H}$ a complete intersection of GK type
in the sense of \cite{KKZ4}?
\end{question}
This question is still open, even for $H = \Bbbk^G$.

 One key idea in this paper is to relate
the covariant ring with the so-called {\it Hasse} algebra (Definition \ref{xxdef2.3})
of the group $G$; this algebra has been considered before in the case that $G$ is a Coxeter group, and was called the nilCoxeter algebra by Fomin-Stanley in \cite{FS}. 

A secondary goal of this paper is to prove some rigidity results
for group coactions on some families of AS regular 
algebras. A remarkable rigidity theorem of Alev-Polo \cite[Theorem 1]{AP}
states:
Let ${\mathfrak g}$ and ${\mathfrak g}'$ be two semisimple Lie algebras. Let $G$ be 
a finite group of algebra automorphisms of $U({\mathfrak g})$ such that 
$U({\mathfrak g})^G \cong  U({\mathfrak g}')$. Then $G$ is trivial and 
${\mathfrak g} \cong {\mathfrak g}'$.
They also proved a rigidity theorem for the Weyl algebras
\cite[Theorem 2]{AP}.
The authors extended Alev-Polo's rigidity theorems to the 
graded case \cite[Theorem 0.2 and Corollary 0.4]{KKZ1}. 
For group coactions, we make the following definition.

\begin{definition}
\label{xxdef0.8}
Let $A$ be a connected graded algebra. We say that $A$ is 
{\it rigid with respect to group coaction} if for every 
finite group $G$ coacting on $A$ homogeneously and inner
faithfully, $A^{\Bbbk^G}$ is NOT isomorphic to $A$ as algebras.
\end{definition}

Our main result concerning the rigidity of the dual group action is the following result, which provides
a dual version of the rigidity results proved in 
\cite[Theorem 0.2 and Corollary 0.4]{KKZ1}. The proof is based on the 
structure results Theorems \ref{xxthm0.3} and \ref{xxthm0.5}.

\begin{theorem}
\label{xxthm0.9}
Let $\Bbbk$ be an algebraically closed field. 
The following AS regular algebras are rigid with respect to 
group coactions.
\begin{enumerate}
\item[(1)]
The homogenization of the universal
enveloping algebra of a finite dimensional semisimple Lie algebra 
$H({\mathfrak g})$.
\item[(2)]
The Rees ring of the Weyl algebra $A_n(\Bbbk)$ with respect to the standard
filtration.
\item[(3)]
The non-PI Sklyanin algebras of global dimension at least 3.
%\item[(4)]
%Noetherian graded Down-up algebras.
\end{enumerate}
\end{theorem}

\begin{remark}
\label{xxrem0.10}
By using ideas and results in \cite{EW2}, one can show that 
there is no nontrivial semisimple Hopf algebra actions on the 
algebras listed in Theorem \ref{xxthm0.9}. Combining
\cite[Theorem 0.2 and Corollary 0.4]{KKZ1} with Theorem \ref{xxthm0.9}, one obtains that
each algebra $A$  in Theorem \ref{xxthm0.9} is rigid with respect 
to semisimple Hopf algebra actions, or, equivalently, if $H$ is a 
semisimple Hopf algebra acting on $A$ inner faithfully and 
$A^H\cong A$, then $H=\Bbbk$.  
\end{remark}

Based on the above remark, we have an immediate question:

\begin{question}
\label{xxque0.11} 
Let $A$ be an algebra that is rigid with respect to finite group 
actions, see \cite{KKZ1}. Is $A$ rigid with respect to 
any semisimple Hopf algebra action? 
\end{question}

This paper is organized as follows. We provide background material 
on Artin-Schelter regular algebras, the Nakayama automorphism, and local 
cohomology  in Section 1. We study the Hasse algebra and Poincar{\'e}
polynomials of a finite group in Section 2. 
In Section 3, we prove some basic properties concerning 
dual reflection groups. The results about the Nakayama automorphisms
are proved in Section 4, and the proofs of Theorems \ref{xxthm0.3}, \ref{xxthm0.4}, \ref{xxthm0.5}, and \ref{xxthm0.6} appear at the end of Section 4.  Theorem \ref{xxthm0.9}, about the rigidity of the dual group action, 
is proved in Section 5.

\section{Preliminaries}
\label{xxsec1}

An algebra $A$ is called {\it connected graded} if
$$A=\Bbbk \oplus A_1\oplus A_2\oplus \cdots$$
and $A_iA_j\subseteq A_{i+j}$ for all $i,j\in {\mathbb N}$.
The Hilbert series of $A$ is defined to be
$$H_A(t)=\sum_{i\in {\mathbb N}} (\dim_{\Bbbk} A_i)t^i.$$
The algebras that we use to replace commutative polynomial rings
are the AS regular algebras \cite{AS}. We recall the definition
below.

\begin{definition}
\label{xxdef1.1}
A connected graded algebra $A$ is called {\it Artin-Schelter Gorenstein} 
(or {\it AS Gorenstein}, for short) if the following conditions hold:
\begin{enumerate}
\item[(a)]
$A$ has injective dimension $d<\infty$ on
the left and on the right,
\item[(b)]
$\Ext^i_A(_A\Bbbk,_AA)=\Ext^i_{A}(\Bbbk_A,A_A)=0$ for all
$i\neq d$, and
\item[(c)]
$\Ext^d_A(_A\Bbbk,_AA)\cong \Ext^d_{A}(\Bbbk_A,A_A)\cong \Bbbk(\bfl)$ for some
integer $\bfl$. Here $\bfl$ is called the {\it AS index} of $A$.
\end{enumerate}
If in addition,
\begin{enumerate}
\item[(d)]
$A$ has finite global dimension, and
\item[(e)]
$A$ has finite Gelfand-Kirillov dimension,
\end{enumerate}
then $A$ is called {\it Artin-Schelter regular} (or {\it AS
regular}, for short) of dimension $d$.
\end{definition}

Let $M$ be an $A$-bimodule, and let $\mu, \nu$ be algebra 
automorphisms of $A$. Then ${^\mu M^\nu}$ denotes the 
induced $A$-bimodule such that ${^\mu M^\nu}=M$ 
as a $\Bbbk$-space, and where
$$a * m * b=\mu(a)m\nu(b)$$
for all $a,b\in A$ and $m\in {^\mu M^\nu}(=M)$.
Let $1$ be the identity of $A$. We also use
${^\mu M}$ (respectively, ${M^\nu}$)
for ${^\mu M^1}$ (respectively, ${^1 M^\nu}$).

Let $A$ be a connected graded finite dimensional algebra. We say 
$A$ is a {\it Frobenius} algebra if there is a nondegenerate 
associative bilinear form 
$$\langle -,- \rangle: A\times A \to \Bbbk,$$
which is graded of degree $-\bfl$. This is equivalent to the 
existence of an isomorphism $A^*\cong A(-\bfl)$ as graded left 
(or right) $A$-modules. There is a (classical) graded Nakayama 
automorphism $\mu\in \Aut(A)$ such that
$\langle a,b \rangle=\langle \mu(b), a\rangle$
for all $a,b\in A$. Further, $A^*\cong {^\mu A^1}(-\bfl)$
as graded $A$-bimodules. 
A connected graded AS Gorenstein algebra of injective dimension 
zero is exactly a connected graded Frobenius algebra. 
The Nakayama automorphism is also
defined for certain classes of infinite dimensional algebras;
see the next definition.

\begin{definition}
\label{xxdef1.2}
Let $A$ be an algebra over $\Bbbk$, and let $A^e = A \otimes A^{op}$.
\begin{enumerate}
\item[(1)]
$A$ is called {\it skew Calabi-Yau} (or {\it skew CY}, for short) if
\begin{enumerate}
\item[(a)]
$A$ is homologically smooth, that is, $A$ has a projective resolution 
in the category $A^e$-Mod that has finite length and such that each 
term in the projective resolution is finitely generated, and
\item[(b)]
there is an integer $d$ and an algebra automorphism $\mu$ of $A$ 
such that
\begin{equation}
\label{E1.2.1}\tag{E1.2.1}
\Ext^i_{A^e}(A,A^e)=\begin{cases} 0 & i\neq d\\
{^1 A^\mu} & i=d,\end{cases}
\end{equation}
as $A$-bimodules, where $1$ denotes the identity map of $A$.
\end{enumerate}
\item[(2)]
If \eqref{E1.2.1} holds for some algebra automorphism $\mu$ 
of $A$, then $\mu$ is called the {\it Nakayama automorphism} 
of $A$, and is usually denoted by $\mu_A$. 
\item[(3)]
We call $A$ {\it Calabi-Yau} (or {\it CY}, for short) if 
$A$ is skew Calabi-Yau and $\mu_A$ is inner (or equivalently, 
$\mu_A$ can be chosen to be the identity map after changing 
the generator of the bimodule ${^1 A^\mu}$).
\end{enumerate}
\end{definition}

If $A$ is connected graded, the above definition should be made 
in the category of graded modules and \eqref{E1.2.1}
should be replaced by
\begin{equation}
\label{E1.2.2}\tag{E1.2.2}
\Ext^i_{A^e}(A,A^e)=\begin{cases} 0 & i\neq d\\
{^1 A^\mu}(\bfl) & i=d,\end{cases}
\end{equation}
where ${^1 A^\mu}(\bfl)$ is the shift of ${^1 A^\mu}$ by degree $\bfl$.

We will use local cohomology later.
Let $A$ be a locally finite ${\mathbb N}$-graded algebra and 
$\fm$ be the graded ideal $A_{\geq 1}$. Let $A$-$\GrMod$ denote 
the category of ${\mathbb Z}$-graded left $A$-modules. For 
each graded left $A$-module $M$, we define
$$\Gamma_{\fm}(M) =\{ x\in M\mid A_{\geq n} x=0 \; {\text{for some $n\geq 1$}}\;\}
=\lim_{n\to \infty} \Hom_A(A/A_{\geq n}, M)$$
and call this the $\fm$-torsion submodule of $M$. It is 
standard that the functor $\Gamma_{\fm}(-)$ is a left 
exact functor from $A$-$\GrMod$ to itself. Since this category 
has enough injectives, the right derived functors 
$R^i\Gamma_{\fm}$ are defined and called the local 
cohomology functors. Explicitly, one has 
$$R^i\Gamma_{\fm}(M)=\lim_{n\to \infty} \Ext^i_A(A/A_{\geq n}, M).$$ 
See \cite{AZ, VdB} for more details.

The Nakayama automorphism of a noetherian AS regular algebra can be 
recovered by using local cohomology \cite[Lemma 3.5]{RRZ2}:
\begin{equation}
\label{E1.2.3}\tag{E1.2.3}
R^d \Gamma_{\fm} (A)^*\cong {^\mu A^1}(-\bfl)
\end{equation}
where $\bfl$ is the AS index of $A$.

\section{Poincar{\'e}  polynomials and Hasse algebras}
\label{xxsec2}

Let $G$ be a finite group, though some of the definitions make
sense in the infinite case. Let $e$ be the unit of $G$. 
We say $\Re\subseteq G$ is a set of generators of $G$, if $\Re$ 
generates $G$ and $e\not\in \Re$ ($\R$ need not be a minimal set of generators of $G$). We recall some definitions.

\begin{definition}
\label{xxdef2.1}
Let $\Re$ be a set of generators of a group $G$.
The {\it length} of an element $g\in G$ with respect to 
$\Re$ is defined to be 
$$\lrr(g):=\min\{ n \mid v_1\cdots v_n=g, \; {\text{for some $v_i\in \Re$}}\}.$$
We define the length of the identity $e$ to be 0.  The length $\lrr(g)$ is also called {\it reduced length} by some authors.
\end{definition}

\begin{definition}
\label{xxdef2.2}
Let $\Re$ be a set of generators of $G$.
The {\it Poincar{\'e} polynomial} associated to $\Re$ is defined to be
$$\fp_\Re(t)=\sum_{g\in G} t^{\lrr(g)}.$$
\end{definition}

The notion of length of a group element with respect to $\Re$, 
where $\Re$ is the set of Coxeter generators, is standard for 
Coxeter groups \cite[p.15]{BjB}, and its generating function 
is called the Poincar\'{e} polynomial in \cite[p. 201]{BjB}.  
There is also interest in the generating function for other groups 
(e.g. the alternating group \cite[p. 3]{Ro} and \cite[p. 849]{BRR});  
the article \cite{He} gives a survey of recent progress in determining 
the order of magnitude of the maximal length of an element of a finite 
group $G$ with respect to {\it any} generating set of $G$ 
(i.e. the maximal degree of any Poincar\'{e} polynomial of the group) 
for many families of linear algebraic groups and permutation groups 
(often with respect to $\Re  \cup \Re^{-1}$).

If there is no confusion (for example, $\Re$ is fixed), 
we might use $\fp(t)$ instead of $\fp_\Re(t)$.

\begin{definition}
\label{xxdef2.3} Let $\Re$ be a set of generators of a group $G$.
\begin{enumerate}
\item[(1)]
The {\it Hasse algebra} associated to $\Re$, denoted
by ${\mathcal H}_G(\Re)$,  is the associated 
graded algebra of the group algebra $\Bbbk G$ with respect to the generating
space $\Bbbk e+\Bbbk \Re$,
$${\mathcal H}_G(\Re):=\bigoplus_{i=0}^{\infty} 
(\Bbbk e+\Bbbk \Re)^i/(\Bbbk e+\Bbbk \Re)^{i-1}$$
where $(\Bbbk e+\Bbbk \Re)^0=\Bbbk $ and $(\Bbbk e+\Bbbk \Re)^{-1}=0$.
\item[(1')]
Equivalently, the {\it Hasse algebra} associated to $\Re$, denoted
by ${\mathcal H}_G(\Re)$, is the
associative algebra with $\Bbbk$-linear basis $G$ together with 
multiplication determined by 
$$ g\cdot h=\begin{cases} gh & \lrr(gh)=\lrr(g)+\lrr(h),\\
0& \lrr(gh)<\lrr(g)+\lrr(h).\end{cases}$$
\item[(2)]
A  {\it skew Hasse algebra} associated to $\Re$ is an 
associative algebra with $\Bbbk$-linear basis $G$ together with 
multiplication determined by 
$$ g\cdot h=\begin{cases} \alpha_{g,h} gh & \lrr(gh)=\lrr(g)+\lrr(h),\\
0& \lrr(gh)<\lrr(g)+\lrr(h),\end{cases}$$
where $\{\alpha_{g,h}\mid \lrr(gh)=\lrr(g)+\lrr(h)\}$ is a set of 
nonzero scalars in $\Bbbk$.
\end{enumerate}
\end{definition}

\begin{remark}
\label{xxrem2.4}
\begin{enumerate}
\item[(1)]
It is easy to see that Definitions \ref{xxdef2.3}(1) and \ref{xxdef2.3}(1') 
of the Hasse algebra are equivalent.
\item[(2)]
The first definition can be generalized to any finite dimensional Hopf 
algebra $K$ as follows: Let $V$ be a generating $\Bbbk$-space of the 
associative algebra $K$ with $1\not\in V$ that
is also a right coideal of $K$ (namely, we have $\Delta: V\to V\otimes K$). 
The {\it Hasse algebra} associated to $\Re$ is defined to be
$${\mathcal H}_K(V):=\bigoplus_{i=0}^{\infty} (\Bbbk+V)^i/(\Bbbk+V)^{i-1}.$$
On the other hand, the second definition seems easier to understand 
in the group case.
\item[(3)]
Following the definition, ${\mathcal H}_G(\Re)$ is a connected
graded algebra with $\deg (g)=\lrr(g)$ for all $g\in G$.  
\item[(4)]
The Poincar{\'e} polynomial $\fp_\Re(t)$ is the Hilbert series 
of the graded algebra ${\mathcal H}_G(\Re)$.
\item[(5)]
The Hasse algebra was called the {\it nilCoxeter algebra} by
Fomin-Stanley \cite{FS} for Coxeter group $G$ with 
Coxeter generating set $\Re$. See \cite{Al} for some
related results. %(move this remark to after 2.10?)
\end{enumerate}
\end{remark}

We begin with a few examples.

\begin{example}
\label{xxex2.5} 
Let $G$ be the quaternion group $\{\pm 1, \pm i, \pm j, \pm k\}$
with 
$$ij=k=-ji, \quad i^2=j^2=k^2=-1.$$ Let $\Re$ be a generating set 
$\{x_1:=i, x_2:=j, x_3:=-j\}$. 
Then elements of length two are $\{y_1:=k, y_2:= -k, y_3=-1\}$
and the unique element of length 3 is $z:=-i$. The Poincar{\'e} 
polynomial is $\fp_\Re(t)=1+3t+3t^2+t^3=(1+t)^3$ and the Hasse algebra 
is ${\mathcal H}:={\mathcal H}_G(\Re)=\Bbbk \oplus  \Bbbk x_1\oplus \Bbbk x_2
\oplus \Bbbk x_3\oplus \Bbbk y_1\oplus \Bbbk y_2\oplus \Bbbk y_3\oplus \Bbbk z$ 
with multiplication determined as in Definition \ref{xxdef2.3}(1').
Explicitly, we have

\begin{center}
\begin{tabular}{|c|c|c|c|c|c|c|c|}\hline
multiplication & $x_1$ & $x_2$ & $x_3$ & $y_1$ & $y_2$ & $y_3$ & $z$ 
\\ \hline
$x_1$ &  $y_3$ & $y_1$ & $y_2$ & 0 & 0   & $z$ & 0 \\ \hline
$x_2$ &  $y_2$ & $y_3$ & 0     & 0 & $z$ & 0   & 0\\ \hline
$x_3$ &  $y_1$ & 0     & $y_3$ & $z$ & 0 & 0   & 0\\ \hline
$y_1$ &  0     & $z$   & 0     & 0& 0& 0& 0\\ \hline
$y_2$ &  0     & 0     & $z$   & 0& 0& 0& 0\\ \hline
$y_3$ &  $z$   & 0     & 0     & 0& 0& 0& 0\\ \hline
$z$ &  0 &0  &0  &0  &0  &0  &0 \\ \hline
\end{tabular}
\end{center}

It is easy to see that this is a non-symmetric Frobenius algebra and the Nakayama
automorphism of this algebra is determined by
$$\mu_{\mathcal H}: x_1\mapsto x_1, \; x_2\mapsto x_3, \; x_3\mapsto x_2
$$
which has order 2. 

With some effort one can show that ${\mathcal H}$ is a Koszul algebra
and its Koszul dual ${\mathcal E}$ is the AS regular algebra $\Bbbk\langle a_1,a_2,a_3\rangle
/(a_1a_2+a_3a_1, a_1a_3+a_2a_1, a_1^2+a_2^2+a_3^2)$. Therefore the 
GK-dimension of the Ext-algebra of ${\mathcal H}$ is 3.
\end{example}

\begin{example}
\label{xxex2.6} 
Let $G$ be the symmetric group $S_3$ and $\Re$ be the generating set 
$\{x_1:=(12), x_2:=(23)\}$. 
Then elements of length two are $\{y_1:=(123), y_2:=(132)\}$
and the only element of length 3 is $z:=(13)$. The Poincar{\'e} 
polynomial is $\fp_\Re(t)=1+2t+2t^2+t^3$ and the Hasse algebra 
is ${\mathcal H}:={\mathcal H}_G(\Re)=\Bbbk  \oplus  \Bbbk  x_1
\oplus \Bbbk  x_2\oplus \Bbbk  y_1\oplus \Bbbk  y_2\oplus \Bbbk  z$ 
with multiplication determined by

\begin{center}
\begin{tabular}{|c|c|c|c|c|c|}\hline
multiplication & $x_1$ & $x_2$ & $y_1$ & $y_2$ & $z$ 
\\ \hline
$x_1$ &  0 & $y_1$ &  0   & $z$ & 0 \\ \hline
$x_2$ &  $y_2$ & 0 &  $z$ & 0   & 0\\ \hline
$y_1$ &  $z$    & 0  &  0& 0& 0\\ \hline
$y_2$ &  0      & $z$    &  0& 0& 0\\ \hline
$z$ &  0 &0  &0  &0    &0 \\ \hline
\end{tabular}
\end{center}

This is a non-symmetric Frobenius algebra and the Nakayama
automorphism of this algebra is determined by the map of interchanging $x_1$ 
with $x_2$. It is easy to see that 
${\mathcal H}$ is isomorphic to $k\langle a,b\rangle
/(a^2, b^2, aba-bab)$ which is not a Koszul algebra. 
The Ext-algebra ${\mathcal E}$ of the ${\mathcal H}$, which 
was computed in \cite[Theorem 3.3]{SV}, has 
GK-dimension 2.
\end{example}

\begin{example}
\label{xxex2.7}
Let $G$ be the dihedral group $D_{2n}$ with $2n$ elements, and let 
$s_1 = r$ and $s_2 = r\rho$ be the usual Coxeter generators where $r^2 = \rho^n = e$ and $\rho r = r \rho^{-1}$. 
The Poincar{\'e} polynomial of $G$ with this choice of generating set is
$$\fp_\Re(t)=1+2t+2t^2+\cdots+ 2t^{n-1}+t^{n-1}= (1+t)(1+t+t^2+\cdots + t^{n}).$$
So the Hasse algebra has the following elements in each degree:

degree 1: $r, r\rho$;

degree 2: $\rho^{-1}, \rho$;

degree 3: $r\rho^{-1}, r\rho^2$;

degree 4: $\rho^{-2}, \rho^2$; etc.

When $n$ is odd the unique largest length element is $r\rho^{\frac{n+1}{2}}$,
while  when $n$ is even it is $\rho^{\frac{n}{2}}$, which is central.
Hence, in the case when $n$ is even, ${\mathcal H}_{D_{2n}}(\Re)$ is a symmetric 
Frobenius algebra, but a non-symmetric 
Frobenius algebra when $n$ is odd.
\end{example}

Based on the above examples, here are some natural questions.

\begin{question}
\label{xxque2.8}
\begin{enumerate}
\item[(1)]
When is ${\mathcal H}_G(\Re)$ a Frobenius algebra?
\item[(2)]
When is ${\mathcal H}_G(\Re)$ a symmetric Frobenius algebra?
\item[(3)]
When is ${\mathcal H}_G(\Re)$ a Koszul (or $N$-Koszul) algebra?
\item[(4)]
Let ${\mathcal E}_{G}(\Re)$ be the Ext-algebra of the connected graded 
algebra ${\mathcal H}_G(\Re)$.
What is the GK-dimension of ${\mathcal E}_{G}(\Re)$ in terms of $(G,\Re)$? 
\end{enumerate}
\end{question}

We give an answer to the first question in the next theorem. 
A polynomial $p(t)=\sum_{i=0}^n a_i t^i$ 
of degree $n$ 
is called {\it palindrome} if $a_i=a_{n-i}$ for all $i$. If ${\mathcal H}_G(\Re)$ 
is Frobenius, then $\fp_\Re(t)$ is palindrome. As a consequence, the leading 
coefficient of $\fp_\Re(t)$ is 1.

\begin{theorem}
\label{xxthm2.9}
Let $\Re$ be a set of generators of $G$. Suppose that 
$\fp_\Re(t)$ is palindrome. Then the following hold.
\begin{enumerate}
\item[(1)]
${\mathcal H}_G(\Re)$ is Frobenius.
\item[(2)]
Let $\mu$ be the graded Nakayama automorphism of ${\mathcal H}_G(\Re)$.
Then $\mu$ is determined by a permutation of the elements in $G$ of length 1.
\end{enumerate}
\end{theorem}

\begin{proof} Since $\fp_\Re(t)$ is palindrome, the coefficient 
of the leading term of $\fp_\Re(t)$ is 1. This means that there
is a unique element in $G$ of the maximal length. Let $\cdot$ 
be the multiplication of ${\mathcal H}_G(\Re)$. The group product
is suppressed.

Let $m\in G$ be the unique element of the maximal 
length $d$. (Later we will see that $m$ agrees with the mass element
in Definition \ref{xxdef0.2}.)
Let $\{x_{i1},\cdots x_{in_i}\}$ be the complete list of 
elements in $G$ of length $i$ for $1\leq i\leq d-1$. We now prove the 
following {\bf claim}: 
\begin{enumerate}
\item[(a)]
For each $x_{ij}$, there is a unique element of length 
$s:=d-i$, say $x_{sw}$, such that $x_{ij}\cdot x_{sw}=x_{ij} x_{sw}=m$ 
where $x_{sw}=x_{ij}^{-1}m$ as in $G$. 
\item[(b)]
If $t\neq w$, then $x_{ij}\cdot x_{st}=0$.
\item[(c)]
Every element 
of length $d-i$ is of the form $x_{ij}^{-1}m$ (for different $j$). 
\end{enumerate}
Consequently, for any  $t\neq w$ (where $x_{sw}$ is defined as in (a)), 
$x_{ij}\cdot x_{st}=0$. 
We prove the claim by induction on $i$. Initial case $i=1$:
For each fixed $j\leq n_1$, we consider $x_{1j}^{-1} m$. Since $x_{1j}\neq e$,
$x_{1j}^{-1}m\neq m$. This means that $\lrr(x_{1j}^{-1}m)\leq d-1$.
Since $\lrr(m)=d$ and $\lrr(x_{1j})=1$, $\lrr(x_{1j}^{-1}m)\geq d-1$. Thus
$\lrr(x_{1j}^{-1}m)= d-1$ and $x_{1j}^{-1}m=x_{d-1w}$ for some 
$w\leq n_{d-1}$. We have $x_{1j}\cdot x_{d-1w}=x_{1j}x_{d-1w}=m$.
For any $t\neq w$, $x_{1j}x_{d-1t}\neq m$, so 
$x_{1j}\cdot x_{d-1 t}=0$ by definition. Since $\fp_\Re(t)$ is 
palindrome, every element of length $d-1$ is of the form $x_{1j}^{-1}m$. 
Inductive step: Let $i\geq 1$. Suppose the claim holds for $i'\leq i$.
Pick a $x_{i+1j}$ for any fixed $j$, then $x_{i+1j}^{-1}m$ is not equal to any 
$x_{i'j'}^{-1}m$ for all $i'\leq i$. Therefore $x_{i+1j}^{-1}m$ is not
of the form $x_{d-i'j'}$ for all $i'\leq i$ and all $j'\leq n_{d-i'}$.
As a consequence, $\lrr(x_{i+1j}^{-1}m)\leq d-i-1$. Since 
$m=x_{i+1j}(x_{i+1j}^{-1}m)$, $\lrr(x_{i+1j}^{-1}m)\geq d-i-1$. 
Thus $\lrr(x_{i+1j}^{-1}m)=d-i-1$. Write $x_{i+1j}^{-1}m$ as $x_{d-i-1w}$.
For any $t\neq w$, $x_{i+1j} x_{d-i-1t}\neq m$. By definition,
$x_{i+1j}\cdot x_{d-i-1t}=0$. Since $\fp_\Re(t)$ is palindrome, every element 
of length $d-(i+1)$ is of the form $x_{(i+1)j}^{-1}m$ for some $j$. 
Therefore we proved the claim. 

(1) Recall that, if there is a nondegenerate associative bilinear form
$$\langle -, -\rangle: A\times A \to k$$ 
and an algebra automorphism $\mu$ of $A$ 
such that 
$$\langle a,b\rangle=\langle \mu(b), a\rangle(=\langle b, \mu^{-1}(a)\rangle)$$ 
for all $a,b\in A$,
then $A$ is a Frobenius algebra and $\mu$ is a Nakayama automorphism of $A$. 
The Nakayama
automorphism always exists for Frobenius algebras and the 
graded Nakayama automorphism is unique for 
a connected graded Frobenius algebra. 

Note that ${\mathcal H}_G(\Re)$ is graded with $m$ having 
the highest degree. Let ${\text{pr}}_{\Bbbk m}$ be the projection to 
the highest degree component of ${\mathcal H}_G(\Re)$, and define $\langle a,b\rangle={\text{pr}}_{\Bbbk m}(ab)$.
Then the {\bf claim} implies that $\langle a,b\rangle$ is a nondegenerate
associative bilinear form. Therefore ${\mathcal H}_G(\Re)$ is Frobenius.

(2) 
By the {\bf claim}, there is a permutation $\sigma\in S_{n_1}$ such that
\begin{equation}
\label{E2.9.1}\tag{E2.9.1}
x_{1j} \cdot x_{d-1 w}=m =x_{d-1 w} \cdot x_{1\sigma(j)}
\end{equation}
for all $1\leq j\leq n_1$. Using this observation we see  that 
$\mu^{-1}$ maps $x_{1j}$ to $x_{1\sigma(j)}$. Since
${\mathcal H}_G(\Re)$ is generated in degree 1, $\mu$ 
is completely determined by
$\sigma^{-1}$.
\end{proof}

The permutation given in Theorem \ref{xxthm2.9}(2) 
is called the {\it Nakayama permutation} of ${\mathcal H}_G(\Re)$.
In the setting of Theorem \ref{xxthm2.9}, there is the unique element 
in $G$ of the maximal length with respect to the length function $l$. 
We give an answer to Question \ref{xxque2.8}(2) next.

\begin{theorem}
\label{xxthm2.10} Suppose $\fp_\Re(t)$ is palindrome. 
Let $m$ be the unique element of the maximal length. 
Let $\mu$ be the graded Nakayama automorphism of 
${\mathcal H}_G(\Re)$.
\begin{enumerate}
\item[(1)]
$\mu$ permutes the elements of $G$ and preserves 
the length.
\item[(2)]
$\mu^{-1}$ equals the conjugation $\eta_{m}: g\to m^{-1} g m$
when applied to the basis element $g\in G$.
\item[(3)]
${\mathcal H}_G(\Re)$ is symmetric if and only if $m$ is central.
\item[(4)]
Suppose $\fp_\Re(t)=1+a t+ \sum_{i\geq 2} a_i t^i$. Then 
$m^n$ is central for some $n$ dividing $a!$.
\item[(5)] 
The conjugation $\eta_{m}$ in part {\rm{(2)}} preserves the length of 
$g\in G$.
\end{enumerate}
\end{theorem}

\begin{proof}
(1) By the proof of Theorem \ref{xxthm2.9}, 
$\mu^{-1}$ maps $x_{ij}$ to $x_{ij'}$ where $j'$ is determined by
\begin{equation}
\label{E2.10.1}\tag{E2.10.1}
x_{ij} \cdot x_{d-i w}=x_{ij} x_{d-i w} =m =x_{d-i w} x_{i j'}=x_{d-i w} 
\cdot x_{i j'}.
\end{equation}
Hence the assertion follows. Since ${\mathcal H}_G(\Re)$ is generated in degree 1, 
this automorphism $\mu^{-1}$ is completely determined by \eqref{E2.9.1}.

(2) 
By part (1) and \eqref{E2.10.1}, for any $(i,j)$, 
$$\mu^{-1}(x_{ij})=x_{ij'}=
x_{d-iw}^{-1} m=m^{-1} (m x_{d-iw}^{-1})m=m^{-1} x_{ij} m=
\eta_{m}(x_{ij})$$
which implies the assertion.

(3) This follows from the fact $\mu^{-1}(=\eta_{m})$ is the identity
if and only if $m$ is central.

(4) Restricted to the degree 1 component of ${\mathcal H}_G(\Re)$,
$\eta_{m}$ is a permutation in $S_a$. So $\eta_{m}$ has order
$n$ where $n$ divides $a!$. This implies that $\eta_{m^n}$
is the identity, which is equivalent to the assertion that
$m^n$ is central.

(5) This is an immediate consequence of parts (1,2). 
\end{proof}

When $G$ is a Coxeter group and $(G,\Re)$ a Coxeter
system, the Hasse algebra ${\mathcal H}_G(\Re)$ agrees with the
{\it nilCoxeter algebra} of Fomin-Stanley \cite{FS}, which  was 
studied by several people;  see, for example, \cite{Al, Ba, KM}. 
The Frobenius property of the nilCoxeter algebras was proved by
\cite{Al} in which some statements in Theorems \ref{xxthm2.9} 
and \ref{xxthm2.10} are proven (for the special case of 
nilCoxeter algebras).
There are other Frobenius algebras associated 
to $(G,\Re)$, for example, the covariant algebra of $(G,\Re)$ and 
the Nichols algebra (or the Nichols-Woronowicz algebra as it is called 
in \cite{Ba, KM, MS}) ${\mathcal B}(V)$. 

\begin{theorem}
\label{xxthm2.11}
If the Hasse algebra ${\mathcal H}_G(\Re)$ is Frobenius, then so is 
any skew Hasse algebra associated to $(G,\Re)$.
\end{theorem}

\begin{proof} Since ${\mathcal H}_G(\Re)$ is Frobenius,
$\fp_\Re(t)$ is palindrome. Then all statements made in the 
proof of Theorem \ref{xxthm2.9} hold. 

Let $A$ be a skew Hasse algebra associated to $(G,\Re)$. Then 
$A$ is a connected graded algebra with a basis $\{g\mid g\in G\}$.
Following the proof of Theorem \ref{xxthm2.9}, for each $i$, let 
$\{x_{i1},\cdots x_{in_i}\}\subset G$ be a $\Bbbk$-linear basis 
of $A_i$, which is the complete list of elements in $G$ of length $i$. 
Then, by the definition of skew Hasse algebra, the following 
hold in the algebra $A$:
\begin{enumerate}
\item[(1)]
For each $x_{ij}$, there is a unique element of length 
$s:=d-i$, say $x_{sw}$, such that $x_{ij}\cdot x_{sw}=x_{ij} x_{sw}=
\alpha m$ where $x_{sw}=x_{ij}^{-1}m$ is in $G$ and $\alpha$ is a nonzero
scalar dependent on $x_{ij}$ and $x_{sw}$. 
\item[(2)]
If $t\neq w$, then $x_{ij}\cdot x_{st}=0$.
\item[(3)]
Every element 
of length $d-i$ is of the form $x_{ij}^{-1}m$ (for different $j$). 
\end{enumerate}
Let ${\text{pr}}_{\Bbbk m}$ be the projection from $A$ to 
the highest degree component and define $\langle a,b\rangle={\text{pr}}_{\Bbbk m}(ab)$.
It follows from (1,2,3) that $\langle a,b\rangle$ is a nondegenerate
associative bilinear form. Therefore $A$ is Frobenius.

By arguments similar to those in the proofs of Theorems \ref{xxthm2.9} and \ref{xxthm2.10}, we see that
the graded Nakayama automorphism of $A$ is of the form
\begin{equation}
\label{E2.11.1}\tag{E2.11.1}
\mu_A: g\to \beta(g) mg m^{-1}, \quad \forall \; g\in G
\end{equation}
where $\{\beta(g)\mid g\in G\}$ are nonzero scalars.
\end{proof}

\section{Dual reflection groups}
\label{xxsec3}

A commutative algebra is AS regular if and only if it is a polynomial ring, so we
can extend the notion of a reflection group to groups acting on more general AS regular algebras using the following definition.

\begin{definition}
\label{xxdef3.1}
A finite group $G$ is called a {\it  reflection group} 
(in the noncommutative setting) if $\Bbbk G$ acts homogeneously and 
inner faithfully on a noetherian AS regular domain $A$ that is 
generated in degree 1 such that the fixed subring $A^G$ is AS 
regular.
\end{definition}

When $A$ is a skew polynomial ring, a special AS regular algebra,
reflections (or quasi-reflections as they were called in  \cite[Definition 2.2]{KKZ1}) and 
reflection groups were studied in \cite{KKZ1, KKZ2}. One interesting class 
of reflection groups that arises in this context are the so-called mystic reflection groups \cite{KKZ2}.

More generally, we can extend the definition of reflection group to semisimple Hopf algebras in the following definition.

\begin{definition}
\label{xxdef3.2}
Let $H$ be a semisimple Hopf algebra. We say $H$ is a {\it 
reflection Hopf algebra} if $H$ acts homogeneously and 
inner faithfully on 
a noetherian AS regular domain $A$ that is 
generated in degree 1 such that the fixed
subring $A^H$ is again AS regular. In this case we say
that {\it $H$ acts on $A$ as a reflection Hopf algebra}. 
\end{definition}

In \cite[Examples 7.4 and 7.6]{KKZ3} we show that $H_8$, the 
Kac-Palyutkin Hopf algebra of dimension 8 (a semisimple Hopf 
algebra that is not a isomorphic to a group algebra), is a 
reflection Hopf algebra for the skew polynomial algebras 
$\Bbbk_{-i}[u,v]$ and $\Bbbk_{-1}[u,v]$. Further results on 
the structure and properties of reflection Hopf algebras
are the object of research in progress.

Comparing Definition \ref{xxdef3.1} with Definition \ref{xxdef3.2},
it is clear that a reflection group is a special case of a 
reflection Hopf algebra.  The main object 
of this paper is the {\it dual reflection group} in 
Definition \ref{xxdef0.1}, which is 
another spacial case of Definition \ref{xxdef3.2}. 
Let $G$ be a finite group and let $H$ be the dual Hopf algebra 
$\Bbbk^G:=(\Bbbk G)^\circ$. We say $G$ is a {\it dual reflection group} 
if $H$ acts homogeneously and inner faithfully on a noetherian 
AS regular domain $A$ that is generated in degree 1 such that 
the fixed subring $A^H$ is AS regular [Definition \ref{xxdef0.1}]. 
In this case we say $G$ coacts on $A$ as a {\it dual reflection group}. 
The group-theoretic or 
combinatorial connection between reflection groups and dual 
reflection groups is not yet evident;  Example \ref{xxex3.7} is one example of a dual reflection group, and \cite{KKZ5} will contain further examples.

We note that the study of $\Bbbk^G$-actions on noncommutative 
algebras is an important and interesting topic related to different 
areas of mathematics even if $G$ is not a dual reflection group and $A$ 
is not an AS regular algebra \cite{CM}. 

The algebraic and coalgebraic structure of $\Bbbk^G$ were reviewed
in introduction. 

When $A$ is $G$-graded, we can write 
$A$ as $\bigoplus_{g\in G} A_g$.
The following lemma  is well-known.

\begin{lemma}
\label{xxlem3.3} Let $A$ be a noetherian AS regular domain and 
$H$ be a semisimple Hopf algebra acting homogeneously and inner 
faithfully on $A$.
\begin{enumerate}
\item[(1)]\cite[Lemma 2.4]{KKZ3}
The fixed subring $A^H$ is noetherian and $A$ is finitely generated 
over $A^H$ on the left and on the right.
\item[(2)]%\cite[Lemma 1.11(d)]{KKZ1}
Suppose that the fixed subring $A^H$ is AS regular.
Then $\gldim A^H=\gldim A$ and $A$ is free over $A^H$ on both sides.
\end{enumerate}
\end{lemma}

\begin{proof} (2) The statement of \cite[Lemma 1.10(d)]{KKZ1} is
for $H=\Bbbk G$, but the proof of \cite[Lemma 1.10(d)]{KKZ1} works for 
any semisimple Hopf algebra $H$.
\end{proof}

We need another general result.

\begin{lemma}
\label{xxlem3.4} 
Let $A$ be a noetherian AS Gorenstein algebra and $B$ be a 
subalgebra of $A$. Assume that
\begin{enumerate}
\item[(i)]
$B$ is AS Gorenstein.
\item[(ii)]
$A$ is finitely generated and free over $B$ on both sides.
\end{enumerate}
Then 
\begin{enumerate}
\item[(1)]
$\injdim A=\injdim B$.
\end{enumerate}
Suppose further that
\begin{enumerate}
\item[(iii)]
$AB_{\geq 1}=B_{\geq 1} A$.
\end{enumerate}
Let $C=A/I$ where $I$ is the 2-sided ideal $AB_{\geq 1}$. 
Then 
\begin{enumerate}
\item[(2)]
$C$ is Frobenius. 
\item[(3)]
$\bfl_C=\bfl_A-\bfl_B<0$.
\item[(4)]
$\mu(I)=I$ where $\mu$ is the Nakayama automorphism of
$A$.
\end{enumerate}
\end{lemma}

\begin{proof} 
By the faithful flatness, (ii) implies that $B$ is noetherian.

(1) Let $\mu$ and $\nu$ be the Nakayama automorphisms of $A$ and
$B$ respectively. Let $d_A$ and $d_B$ be the  injective dimensions 
of $A$ and $B$, respectively (and $\bfl_A$ and $\bfl_B$ be the
AS indices) as given in Definition \ref{xxdef1.1}.
By \cite[Lemma 3.5]{RRZ2},
${^\mu A^1}[d_A](-\bfl_A)$ and ${^\nu B^1}[d_B](-\bfl_B)$ are 
rigid dualizing complexes over $A$ and $B$, respectively. 
By \cite[Theorem 3.2(i) and Proposition 3.9(i)]{YZ} 
$${^\mu A^1}[d_A](-\bfl_A)\cong 
{\rm{RHom}}_B(A, {^\nu B^1}[d_B](-\bfl_B))$$
as $(A,B)$-bimodule complexes. Since $A$ is free as a left $B$-module
(and as a right $B$-module),
the homology of ${\rm{RHom}}_B(A, {^\nu B^1}[d_B](-\bfl_B))$ is 
concentrated in complex degree $-d_B$. As a consequence,
$d_A=d_B$ (that is part (1)) and
\begin{equation}
\label{E3.4.1}\tag{E3.4.1}
{^\mu A^1}(-\bfl)\cong \Hom_B(A, {^\nu B^1})
\end{equation}
where $\bfl:=\bfl_A-\bfl_B$.

Let $C=A/B_{\geq 1}A$, which is a graded factor ring of $A$
if $B_{\geq 1}A=AB_{\geq 1}$. Let $D=A/AB_{\geq 1}$.
Then $A\otimes_B \Bbbk\cong D$ as 
left $A$-module and $\Bbbk\otimes_B A\cong C$ as a right $A$-module.

Since $A$ is a free left $B$-module, $\Hom_B(A, {^\nu B^1})$
is a free right $B$-module. Then we have 
quasi-isomorphisms of complexes of left $A$-modules,
$$\begin{aligned}
\Hom_B(A, {^\nu B^1})\otimes_B \Bbbk&\cong 
{\rm{RHom}}_B(A, {^\nu B^1})\otimes^L_B \Bbbk\\
&\cong {\rm{RHom}}_B(A, {^\nu B^1}\otimes^L_B \Bbbk)\\
&\cong {\rm{RHom}}_B(A, \Bbbk)\cong \Hom_B(A,\Bbbk)\\
&\cong \Hom_B(A, \Hom_\Bbbk({_\Bbbk \Bbbk_B},\Bbbk))\\
&\cong \Hom_{\Bbbk}(\Bbbk\otimes_B A, \Bbbk)\cong \Hom_\Bbbk(C,\Bbbk)\\
& = C^*,
\end{aligned}
$$
where the second $\cong$ follows from the fact that $_BA$ is 
finitely generated free over $B$. 

(2,3,4)
Now we assume that $AB_{\geq 1}=B_{\geq 1}A=:I$ is a 2-sided 
ideal of $A$. In this case $C=D$.
Since $I$ is a 2-sided ideal of $A$, so is 
$\mu^{-1}(I)$. Then 
$${^\mu A^1}\otimes_B \Bbbk \cong {^\mu (A/I)^1}\cong A/(\mu^{-1}(I))$$
as left $A$-modules. By applying $-\otimes_B \Bbbk$ to \eqref{E3.4.1}, the 
above computation shows that
$$A/(\mu^{-1}(I))(-\bfl)
\cong {^\mu A^1}(-\bfl)\otimes_B \Bbbk
\cong \Hom_B(A, {^\nu B^1})\otimes_B \Bbbk
\cong C^*$$
as left $A$-modules. As a consequence, $\mu^{-1}(I)=I$ (which is 
part (4)) and 
\begin{equation}
\label{E3.4.2}\tag{E3.4.2}
C(-\bfl)\cong C^*
\end{equation}
as left $C$-modules. Now \eqref{E3.4.2} implies that
$C$ is Frobenius and $\bfl_C=\bfl=\bfl_A-\bfl_B$. Therefore
parts (2,3) follow.
\end{proof}

When $H=\Bbbk^G$ acts on $A$, then the
fixed subring $A^H$ is the group graded $e$-component of $A$, denoted by
$A_e$.

\begin{theorem}
\label{xxthm3.5} Let $A$ be a noetherian AS regular domain.
In parts {\rm{(3-5)}} we further assume that $A$ is generated in 
degree 1. Suppose $G$ coacts on $A$ as a dual reflection group.
Let $H=\Bbbk^G$. 
\begin{enumerate}
\item[(1)]
There is a set of homogeneous elements $\{ f_{g}\mid g\in G\}\subseteq A$ 
with $f_e=1$ such that $A_g=f_g \cdot A_e=A_e \cdot f_g$ for all $g\in G$. 
As a consequence, the nonzero component of $A_g$ with lowest degree 
has dimension 1.
\item[(2)]
$I:=A (A_e)_{\geq 1}$ is a two-sided ideal and $A/I\cong \bigoplus_{g\in G} 
\Bbbk \overline{f_g}$. As a consequence, the covariant ring $A^{cov\; H}$ is
Frobenius.
\item[(3)]
Suppose that, as an $H$-module, $A_1\cong
\oplus_{g\in G} (\Bbbk p_g)^{n_g}$ where $n_g\geq 0$. If $n_g>0$
and $g\neq e$, then $n_g=1$.
\item[(4)] 
The set $\Re:=\{g\in G\mid n_g>0, g\neq e\}$ generates $G$. 
\item[(5)]
$A^{cov\; H}$ is a skew Hasse algebra associated to the generating set $\Re$.
As a consequence, $A^{cov\; H}$ is a Frobenius algebra 
generated in degree 1. Furthermore, the Hilbert series of
$A^{cov\; H}$ is palindrome and is a product of cyclotomic
polynomials.
\end{enumerate}
\end{theorem}

By part (1), $A=\bigoplus_{g\in G} A_e f_g$, but $f_g$ is not invertible
except for $g=e$. In some sense $A$ is a twisted semi-group 
ring without the basis elements $f_g$ being invertible.
Both part (2) and (5) of the above theorem prove
that $A^{cov\; H}$ is a Frobenius algebra 
generated in degree 1. 

\begin{proof}[Proof of Theorem \ref{xxthm3.5}]
(1) Note that $A=\bigoplus_{g\in G} A_g$ and each $A_g$ is an 
$A_e$-bimodule. Since $A$ is a domain and the $\Bbbk^G$-action on $A$ is inner 
faithful, we obtain that $A_g\neq 0$ for all $g\in G$. By Theorem 
\ref{xxlem3.3}(2), $A$ is free over $A_e$.  Hence each $A_g$ is free 
over $A_e$. Let $a$ be a nonzero element in $A_{g^{-1}}$, then 
$a A_{g}\subseteq A_e$. Hence $A_{g}$ has rank one as a right 
$A_e$-module, which implies that $A_g=f_g \cdot A_e$ for some 
homogeneous element $f_g$. Therefore the lowest degree of the nonzero 
component of $A_g$ is $\deg f_g$. By symmetry, $A_g=A_e \cdot f_g$.
The consequence is clear.

(2) By part (1), $f_g A_e=A_e f_g$. Hence $f_g (A_e)_{\geq 1}=
(A_e)_{\geq 1} f_g$. Therefore
$$I=A (A_e)_{\geq 1}=\bigoplus_{g\in G} f_g (A_e)_{\geq 1}=
\bigoplus_{g\in G}(A_e)_{\geq 1} f_g =(A_e)_{\geq 1} A$$
which shows that $I$ is a 2-sided ideal. It is clear that 
$A/I=\bigoplus_{g\in G} \Bbbk \overline{f_g}$. The consequence 
follows from Lemma \ref{xxlem3.4}(2) since $A^{cov \; H}=A/I$.

(3) If $g\neq e$ and $n_g>0$, then $A_g$ has lowest degree 1. By 
part (1), $n_g=1$. 

(4) Since $A/I$ is generated in degree 1, $G$ is generated by 
$\{g\in G\mid n_g>0, g\neq e\}$.

(5) The first assertion is easy to check and the second
follows from Theorem \ref{xxthm2.11}.

Since $A$ is free over $A_e$, we have 
$A\cong A_e\otimes A^{cov\; H}$ as vector
spaces. Hence $H_A(t)=H_{A^e}(t) H_{A^{cov\; H}}(t)$.
Since both $H(A)^{-1}$ and $H_{A^e}(t)^{-1}$ are products of
cyclotomic polynomials, so is $H_{A^{cov\; H}}(t)$. 
It is well-known that 
every product of cyclotomic polynomials is palindrome.
\end{proof}

Let $R$ be any algebra and $t$ be a normal nonzerodivisor 
element in $R$. The conjugation $\eta_t$ of $R$ is defined to 
be $\eta_t(r)= t^{-1} r t$ for all $r\in R$. The inverse
conjugation $\eta^{-1}_t$ is denoted by $\phi_t$.

By Theorem \ref{xxthm3.5}(1),
$$A=\bigoplus_{g\in G} f_g B=\bigoplus_{g\in G} B f_g$$
where $B=A_e$ is AS regular. Let $\eta_g$ and $\phi_g$ be the
graded algebra automorphisms of $B$ defined by
\begin{equation}
\label{E3.5.1}\tag{E3.5.1}
\eta_g: x\to f_{g}^{-1} x f_g, \quad {\text{and}} \quad 
\phi_{g}:=\eta_{f_g^{-1}}: x\to f_g x f_g^{-1}
\end{equation}
for all $x\in B$. Then we have 
$$f_g x= \phi_g(x) f_g$$ 
for all $x\in B$. Define $c_{g,h}\in B$ such that
\begin{equation}
\label{E3.5.2}\tag{E3.5.2}
 f_g f_h=c_{g,h} f_{gh}
\end{equation}
for all $g,h\in G$. Then we have
$$\phi_{g}\phi_{h}=
\phi_{f_g}\phi_{f_h}=\phi_{f_g f_h}=\phi_{c_{g,h}}\phi_{gh}$$
where $\phi_{c_{g,h}}$ is defined by sending 
$x\to c_{g,h} x c_{g,h}^{-1}$. If $\lrr(gh)=\lrr(g)+\lrr(h)$,
then $c_{g,h}$ is a scalar and $\phi_{c_{g,h}}$ is 
the identity, whence $\phi_{g}\phi_{h}=\phi_{gh}$.

Then associativity shows the following.

\begin{lemma}
\label{xxlem3.6}
Retain the above notation. 
\begin{enumerate}
\item[(1)]
$c_{g,h}$ is a normal element in $B$.
\item[(2)]
The following holds
$$c_{g,h}\phi_{gh}(b) c_{gh,k}=\phi_{g}\circ \phi_{h}(b) \phi_{g}(c_{h,k})
c_{g,hk}=\phi_{c_{g,h}}\circ \phi_{gh}(b) \phi_{g}(c_{h,k})
c_{g,hk}$$
for all $g,h,k\in G$ and $b\in B$.
\end{enumerate}
\end{lemma}

\begin{example}
\label{xxex3.7} The dihedral group $D_{8}$ of order 8 is 
a dual reflection group. Note that $G:=D_{8}$ is generated by
$r$ of order 2 and $\rho$ of order 4 subject to the 
$r\rho=\rho^3 r$. Let $A$ be generated by 
$x,y,z$ subject to the relations
$$\begin{aligned}
zx& = q xz,\\
yx& = a zy,\\
yz& = xy.
\end{aligned}
$$
for any $a\in \Bbbk^{\times}$ and $q^2=1$. It is easy to see
that $A=\Bbbk_{q}[x,z][y; \sigma]$ where $\sigma$ sends $z$ to $x$
and $x$ to $az$. Hence $A$ is an AS regular algebra of global dimension 
3 that is not PI (for generic $a$). Define the
$G$-degree of the generators of $A$ as
$$\deg_G(x)=r, \; \deg_G(y)= r\rho, \; \deg_G(z)=
r\rho^2.$$
It is easy to check that the defining relations of
$A$ are $G$-homogeneous. Therefore $G$ coacts on 
$A$ homogeneously and inner faithfully. 
Since $r^2=e$, $x^2, y^2, z^2$ are in 
the fixed subring $A_e$. By using the 
PBW basis $\{x^i z^j y^k\mid i,j,k\geq 0\}$
of $A$, one can easily check that 
$x^i z^j y^k\in A_e$ if and only if all $i,j,k$
are even. Another straightforward computation
shows that $A_e=\Bbbk[x^2,z^2][y^2, \tau]$
where $\tau=\sigma^2\mid_{\Bbbk[x^2,z^2]}$. Therefore
$A_e$ is AS regular and $G$ is a dual reflection group.  Although other examples of dual reflection groups will appear in \cite{KKZ5} we do not know another $n$ with $D_{2n}$ a dual reflection group.
\end{example}

\begin{example}
\label{xxex3.8}
Let $G$ be the quaternion group of order 8. We claim that $G$ is not
a dual reflection group. 

Suppose to the contrary that $G$ coacts on a noetherian AS regular 
algebra $A$ generated in degree 1 as a dual reflection group. 
Then, by Theorem \ref{xxthm3.5}(5), $A^{cov\; H}$ is a Frobenius 
skew Hasse algebra, and its Hilbert series is a product of 
cyclotomic polynomials and palindrome.

Any generating set $\Re$ of $G$ must contain two elements of order 4, 
so the degree of $p_{\Re}(t)$ is less than or equal to 3.  Since 
$p_{\Re}(1) = 8$ the only possibility is 
$$p_{\Re}(t) = (1+t)^3 = 1 + 3 t +3 t^2 + t^3.$$  
Hence there must be 3 generators, say $x_1, x_2,x_3$, of 
$G$-degree not equal to $e$.  
Without loss of generality (and up to a conjugation and a 
permutation), the only possibility is $x_1 = i, x_2 = j$ 
and $x_3 = -j$. The Hasse algebra with respect to this $\Re$ is given 
in Example \ref{xxex2.5}. Let $B=A^{\Bbbk^G}$. Then as a $B$-module we have
$$A = B \oplus x_1 B \oplus x_2 B \oplus x_3 B \oplus x_1x_2 B 
\oplus x_1x_3 B \oplus x_1^2 B \oplus x_1^3 B$$
$$\deg_G(f_g):
\hspace*{.3in} e \hspace{.3in} i \hspace*{.32in} j\hspace*{.3in} 
-j \hspace*{.35in} k\hspace*{.3in} -k 
\hspace*{.3in} -1 \hspace*{.25in} -i\qquad\qquad\quad $$
where the second line is the $G$-degree of the generator of each
component. Since $j i=i (-j)$, we obtain a relation
$$x_2 x_1= a x_1 x_3$$
for some nonzero scalar $a\in \Bbbk$. Similarly,
the following relations are forced:
$$\begin{aligned}
x_3x_1 &= b x_1 x_2,\\
x_2^2 &= c x_1^2,\\
x_3^2 &= d x_1^2,
\end{aligned}
$$
for scalars $b,c,d\in \Bbbk^{\times}$.

We can rescale $x_2$ to take $c=1$ (giving possibly new $a, b,d$) 
and rescale $x_3$ to make $d=1$ (possibly changing $a,b$).  Since 
$x_1^2$ commutes with $x_2$ and $x_3$, $a^2=b^2=1$ and 
$x_2x_3$ and $x_3x_2$ commute with each other.  Also $x_3^2 = x_3x_1^2 
= bx_1x_2x_1 = ab x_1^2x_3 = ab x_3^3$ so $ab=1$. Combining with
$a^2=b^2=1$, we obtain that either $a=b=1$ or $a=b=-1$.  
There could be other generators $y_j$s of $G$-degree $e$
and the other relations involving $x_i$s and $y_j$s, but we 
will get a contradiction using the subalgebra $S$ of $A$ 
generated by the $x_1, x_2,x_3$.   

We first claim that the Koszul dual of $S$
is infinite dimensional. Using the comments about the 
scalars above, the relations in $S$ are:
$$\begin{aligned}
x_2 x_1 \pm x_1 x_3&=0,\\
x_3x_1 \pm   x_1 x_2&=0,\\
x_2^2 -x_1^2 &=0,\\
x_3^2 - x_1^2&=0.
\end{aligned}
$$
Then Koszul dual (also called the quadratic dual) $S^!$ of $S$ is 
generated by $\widehat{x_1}, \widehat{x_2}, \widehat{x_3}$ 
with relations: 
$$ \begin{aligned}
\widehat{x_2}  \widehat{x_1} \mp   \widehat{x_1}  \widehat{x_3}&= 0,\\ 
\widehat{x_3} \widehat{x_1} \mp \widehat{x_1} \widehat{ x_2} &= 0,\\
\widehat{x_1}^2 +  \widehat{x_2}^2+ \widehat{x_3}^2 &=0,\\
\widehat{x_2} \widehat{x_3} &= 0,\\
\widehat{x_3} \widehat{x_2} &= 0.
\end{aligned}
$$
The quadratic algebra generated by 
$\widehat{x_1}, \widehat{x_2}, \widehat{x_3}$ subject to 
the first three relation is an AS regular algebra of global
dimension 3, denoted by $D$. It is easy to see that 
$x_2x_3+x_3x_2, x_2x_3$ is a sequence of normal elements
in $D$. Thus 
$$\GKdim S^!=\GKdim D/(x_2x_3+x_3x_2, x_2x_3)=1,$$ 
which implies that $S^!$ is infinite dimensional.

Next we consider the natural graded algebra map $f: S\to A$. 
Note that this map is injective when restricted to degree 1. 
By taking the Koszul dual, we have a graded algebra map
$$f^!: A^!\to S^!.$$
Since $f$ is injective in degree 1, $f^!$ is surjective
in degree 1. Since $S^!$ is generated in degree 1, $f^!$
is surjective. We have shown that $S^!$ is infinite
dimensional, so is $A^!$. By \cite[Proposition 1.3.1, p. 7]{PP}
$A^!$ is a subalgebra of $\Ext^*_A(\Bbbk,\Bbbk)$. We obtain 
that $\Ext^*_A(\Bbbk,\Bbbk)$
is infinite dimensional, a contradiction to the fact that $A$ is 
AS regular.
\end{example}

\section{Nakayama automorphisms}
\label{xxsec4}

For any algebra $A$, the Nakayama automorphism of $A$ (if it exists) is denoted
by $\mu_A$. 
In this section we study the interplay between the Nakayama automorphisms
of the algebras ${\mathcal H}_G(\Re)$, $A$, $A\# \Bbbk^G$, $A_e$ and $A^{cov\; \Bbbk^G}$.
We need to use some facts about the local cohomology that were 
reviewed in Section \ref{xxsec1}.

\subsection{Nakayama automorphism of ${\mathcal H}_G(\Re)$ and skew Hasse algebras}
\label{xxsec4.0}
By Theorem \ref{xxthm2.10}(2), we have
\begin{equation}
\label{E4.0.1}\tag{E4.0.1}
\mu_{{\mathcal H}_G(\Re)}: \; 
g\to mgm^{-1}, \quad \forall \; g\in G.
\end{equation}
By \eqref{E2.11.1}, the Nakayama automorphism of a
skew Hasse algebra $B$ is of the form
\begin{equation}
\label{E4.0.2}\tag{E4.0.2}
\mu_{B}: \; 
g\to \beta(g) mgm^{-1}, \quad \forall \; g\in G,
\end{equation}
where $\{\beta(g) \mid g\in G\}$ are nonzero scalars in $\Bbbk$.

\subsection{Nakayama automorphism of $A_e$}
\label{xxsec4.1}

Let $\phi_g$ be defined as in \eqref{E3.5.1}.

\begin{lemma}
\label{xxlem4.1} Let $A$ be a noetherian AS Gorenstein 
algebra of injective dimension $d$, 
and let $\sigma$ be a graded algebra automorphism
of $A$. Let $M$ be an $A$-bimodule. 
\begin{enumerate}
\item[(1)]
$R^i\Gamma_{\fm}(M^{\sigma})=R^i\Gamma_{\fm}(M)^{\sigma}$ for all $i$.
\item[(2)]
$\Hom_{\Bbbk}(M^{\sigma}, \Bbbk)={^\sigma \Hom_{\Bbbk}(M, \Bbbk)}$.
\item[(3)]
Suppose $B$ is an AS Gorenstein subalgebra of $A$ such that
there is a $G$-graded decomposition
$$A=\bigoplus_{g\in G} f_g \cdot B=\bigoplus_{g\in G} B \cdot f_g,$$
for some elements $\{f_g\mid g\in G\}$,
where each $f_g\cdot B(=B\cdot f_g)$ is isomorphic to ${^1 B^{\phi_g}}$. Then 
there is a ${\mathbb Z}\times G$-graded isomorphism of $B$-bimodules
\begin{equation}
\label{E4.1.1}\tag{E4.1.1}{^{\mu_A} A^1} (-\bfl_A)\cong R^d \Gamma_{\fm}(A)^*
\cong \bigoplus_{g\in G} f_g^{\ast} 
\cdot \{^{\phi_g \mu_B} B^1 \}(-\bfl_B)
\end{equation}
where $f_g^*$ is a $B$-central generator of bidegree 
$(-\deg f_g,g^{-1})$. 
%As a consequence,
%$$f_{m}^* f_g=s(g) f_{mg^{-1}}^*$$ for some scalar
%$s(g)\in \Bbbk^\times$.
\item[(4)]
$\mu_A$ maps $B$ to $B$ and 
$$\mu_A\mid_B=\phi_{m}\mu_B
\quad {\text{or}}\quad
\mu_B=(\eta_m \mu_A)\mid_{B}.$$
\end{enumerate}
\end{lemma}

\begin{proof}
(1,2) Straightforward.

(3) By \eqref{E1.2.3}, 
\begin{equation}
\label{E4.1.2}\tag{E4.1.2}
{^{\mu_A} A^1} (-\bfl_A)\cong R^d 
\Gamma_{\fm}(A)^*.
\end{equation} 
Since $A$ is finitely generated over $B$ 
on both sides, $R^i\Gamma_{\fm}(M)=R^i\Gamma_{\fm_B}(M)$ for
all $i$ and all graded $A$-bimodules $M$ \cite{AZ}. Here $\fm_B=
B_{\geq 1}$. By Lemma \ref{xxlem3.4}(1),
$A$ and $B$ have the same injective dimension, say $d$.
Hence we have ${\mathbb Z}\times G$-graded isomorphisms
of $B$-bimodules
$$\begin{aligned}
R^d\Gamma_{\fm}(A)^*
&\cong R^d\Gamma_{\fm_B}(\bigoplus_{g\in G} f_g \cdot B)^*\\
&\cong \bigoplus_{g\in G} 
R^d\Gamma_{\fm_B}( f_g \cdot {^1 B^{\phi_g}})^* \quad {\text{viewing $f_g$ as a
$B$-central generator}}\\
&\cong \bigoplus_{g\in G} f_g^* \cdot R^d\Gamma_{\fm_B}({^1 B^{\phi_g}})^*
\quad {\text{viewing $f_g^*$ as a
$B$-central generator}}\\
&\cong \bigoplus_{g\in G} f_g^* \cdot (R^d\Gamma_{\fm_B}(B)^{\phi_g})^*\\
&\cong \bigoplus_{g\in G} f_g^* \cdot {^{\phi_g}(R^d\Gamma_{\fm_B}(B)^*)}\\
&\cong \bigoplus_{g\in G} f_g^* \cdot {^{\phi_g}(^{\mu_B} B^1)(-\bfl_B)}\\
&\cong \bigoplus_{g\in G} f_g^* \cdot {^{\mu_B\circ \phi_g}B^1}(-\bfl_B).
\end{aligned}
$$
Combining the above with \eqref{E4.1.2}, we obtain 
\eqref{E4.1.1}. 

(4) Let $s$ be the lowest ${\mathbb Z}$-degree element 
in ${^{\mu_A} A^1} (-\bfl_A)$, which corresponds to $f_{m}^*$ 
by \eqref{E4.1.1}. For every $b\in B$, by using
\eqref{E4.1.1},
$$s\mu_A(b) =b\cdot s= s (\mu_B\phi_m)(b)$$
for all $b\in B$. Then $\mu_A(b)=\mu(B)\phi_{m}(b)$ 
for all $b\in B$. The assertion follows. 
\end{proof}

\subsection{Nakayama automorphism of $A\# \Bbbk^G$}
\label{xxsec4.2}
Since $\Bbbk^G$ is commutative, $\mu_{\Bbbk^G}$ is the identity.
By \cite[Theorem 0.2]{RRZ2}, the Nakayama automorphism of
$A\# \Bbbk^G$ is given by
$$\mu_{A\# \Bbbk^G}=\mu_A \# (\mu_{\Bbbk^G} \circ \Xi^l_{\hdet})
=\mu_A \# \Xi^l_{\hdet}$$
where $\hdet$ is the homological determinant of the 
$\Bbbk^G$-action on $A$, and $\Xi^l_{\hdet}$ is the corresponding 
left winding automorphism. We start with a calculation of
the $\hdet$.

\begin{proposition}
\label{xxpro4.2} Retain the above notation.
\begin{enumerate}
\item[(1)]
The homological determinant $\hdet$ of the $\Bbbk^G$ action 
on $A$ is the projection of $\Bbbk^G$ onto $\Bbbk p_{m^{-1}}$ 
where $m$ is the unique maximal length element 
in $G$. As a consequence, $m^{-1}$ is the homological
codeterminant of $G$-coaction on $A$. 
\item[(2)]
$\bfl_{A_e}=\bfl_A+ \deg \fp_\Re(t)$.
\item[(3)]
The left winding automorphism $\Xi^l_{\hdet}$ is the 
left translation $trans^l_{m}$ by $m$.
\end{enumerate}
\end{proposition}

\begin{remark}
\label{xxrem4.3} Retain the notation as in Proposition 
\ref{xxpro4.2}.
\begin{enumerate}
\item[(1)]
Proposition \ref{xxpro4.2}(1) 
asserts that $m$ is in fact the mass element of 
$G$-coaction on $A$ defined in Definition \ref{xxdef0.2}.
\item[(2)]
Proposition \ref{xxpro4.2}(2) follows also from Lemma 
\ref{xxlem3.4}(3).
\end{enumerate}
\end{remark}

\begin{proof}[Proof of Proposition \ref{xxpro4.2}]
(1,2) Let $H=\Bbbk^G$.
By \cite[Proposition 5.3(d)]{KKZ3},
$\hdet$ is determined
by the following:
Let $u$ be a nonzero element in $\Ext^d_A(\Bbbk, A)$ where $d$ is
the injective dimension of $A$. Then there is an induced 
algebra homomorphism 
$\eta': H \to \Bbbk$ satisfying
$$h\cdot u= \eta(h) u$$
for all $h\in H$. The homological determinant $\hdet$ is equal 
to $\eta'\circ S$. Since 
$H=\Bbbk^G$, the $\eta'$ is the projection $pr_g$ from $H$ to $\Bbbk p_g$ for
some $g\in G$. In this case, we just say that $\eta'$
corresponds to a group element $g\in G$.
In fact, when we view $A$ is a $G$-graded algebra, $g$ is the
$G$-degree of $u$. Thus $\hdet=\eta' \circ S=pr_{g^{-1}}$ 
corresponds to the element $g^{-1}$. 

Now we consider $\Ext^d_{A_e}(\Bbbk,A)$. Since $A$ is free over $A_e$
on the left and the right, by the change of rings, there are 
isomorphisms of $G$-graded vector spaces
$$\Ext^d_{A_e}(\Bbbk,A)\cong
\Ext^d_{A}(A\otimes_{A_e} \Bbbk, A)\cong
\Ext^d_{A}(A^{cov\; \Bbbk^G}, A)
\cong (A^{cov\; \Bbbk^G})^* \otimes_\Bbbk \Ext^d_{A}(\Bbbk,A).$$
Since $A=\bigoplus_{g\in G} f_g A_e\cong A_e\otimes A^{cov\; \Bbbk^G}$
as $G$-graded $A_e$-module, $\Ext^d_{A_e}(\Bbbk,A)
\cong \Ext^d_{A_e}(\Bbbk, A_e)\otimes A^{cov\; \Bbbk^G}$. 

By definition, $\Ext^d_{A_e}(\Bbbk,A_e)$ has bidegree $(-\bfl_{A_e}, e)$, while 
$\Ext^d_A(\Bbbk,A)$ has bidegree $(-\bfl_{A}, g)$. Since $\Ext^d_{A_e}(\Bbbk,A)
\cong \Ext^d_{A_e}(\Bbbk, A_e)\otimes A^{cov\; \Bbbk^G}$, the lowest ${\mathbb Z}$-degree
element in $\Ext^d_{A_e}(\Bbbk,A)$ has bidegree $((-\bfl_{A_e}, e)$, and the lowest 
${\mathbb Z}$-degree element in $(A^{cov\; \Bbbk^G})^* \otimes_\Bbbk \Ext^d_{A}(\Bbbk,A)$ 
has bidegree $(-\deg \fp_{\Re}(t), -m)+(-\bfl_A, g)$. Therefore $g=m$ and 
$\bfl_{A_e}=\bfl_A+\deg \fp_\Re(t)$. Hence the assertions follow.

(3) By definition, for any $p_h\in \Bbbk^G$, 
$$\Xi^l_{\hdet} (p_h)= \sum \hdet ((p_h)_1) (p_h)_2=\sum \hdet (p_s) p_{s^{-1}h}
=\hdet(p_{m^{-1}}) p_{mh}=p_{mh}$$
which is the left translation by $m$. 
\end{proof}

We have an immediate consequence.

\begin{corollary}
\label{xxcor4.4}
The Nakayama automorphism of the $A\# \Bbbk^G$ is given by
$\mu_{A\# \Bbbk^G} =\mu_A\# trans^l_{m}$.
\end{corollary}

\begin{proof}
Since $\Bbbk^G$ is semisimple, $S^2$ is the identity. Then 
assertion follows from \cite[Theorem 0.2]{RRZ2} 
and Proposition \ref{xxpro4.2}.
\end{proof}

%By definition of smash product, we have
%$$(1\# p_g)( a f_h\# 1)=af_h\# p_{h^{-1}g}$$
%for all $a\in A_e$, $g,h\in G$.
%
%\begin{lemma}
%The following hold.
%\begin{enumerate}
%\item[(1)]
%$$(f_{g_1}\# p_{h_1})(f_{g_2}\# p_{h_2})=
%\begin{cases} 0 & h_2\neq g_2^{-1}h_1,\\
%c_{g_1,g_2}f_{g_1g_2}\# p_{h_2} & h_2=g_{2}^{-1} h_1.
%\end{cases}$$
%\item[(2)]
%pp
%\item[(3)]
%\end{enumerate}
%\end{lemma}

\subsection{Nakayama automorphism of $A^{cov\; \Bbbk^G}$}
\label{xxsec4.3}
We have an exact sequence of graded algebras
$$1\to A_e\to A\to A^{cov\; \Bbbk^G}\to 1.$$
Also we can describe $A^{cov\; \Bbbk^G}$ as a skew Hasse
algebra, with $\Bbbk$-linear basis $\{\overline{f_g}\mid g\in G\}$, 
and the multiplication of $A^{cov\; \Bbbk^G}$ is determined
by 
$$\overline{f_g}\cdot \overline{f_h}
=\begin{cases} \alpha(g,h)\overline{f_{gh}} & \lrr(gh)=\lrr(g)+\lrr(h),\\
0& \lrr(gh)<\lrr(g)+\lrr(h),\end{cases}$$
where $\{\alpha(g,h)\mid \lrr(gh)=\lrr(g)+\lrr(h)\}$ is a set of 
nonzero scalars in $\Bbbk$, see Definition \ref{xxdef2.3}(2).
Since $\lrr(g)=\lrr(mgm^{-1})$ [Theorem \ref{xxthm2.10}]
and $\lrr(mg^{-1})=\lrr(m)-\lrr(g)$ [Proof of Theorem \ref{xxthm2.9}], 
we have 
$$\begin{aligned}
\overline{f_{mg^{-1}}}\cdot \overline{f_g}
&=\alpha(mg^{-1},g) \overline{f_m}\\
&=\alpha(mg^{-1},g) \alpha(mgm^{-1},mg^{-1})^{-1} 
\overline{f_{mgm^{-1}}}\cdot \overline{f_{mg^{-1}}}.
\end{aligned}
$$
Combining the above with \eqref{E2.11.1}, the Nakayama automorphism 
of $A^{cov\; \Bbbk^G}$ is
\begin{equation}
\label{E4.4.1}\tag{E4.4.1}
\mu: \overline{f_g}\to \beta(g) \overline{f_{mgm^{-1}}}
\end{equation}
with 
\begin{equation}
\label{E4.4.2}\tag{E4.4.2}
\beta(g)=\alpha(mg^{-1},g) \alpha(mgm^{-1},mg^{-1})^{-1}
\end{equation} 
for all $g\in G$.

If $f_m$ is a normal element in $A$ (which is a domain), then we can define
the conjugation automorphism 
\begin{equation}
\label{E4.4.3}\tag{E4.4.3}
\phi_{m}: x\to f_{m} x f_{m}^{-1}, \quad \forall\; x\in A
\end{equation}
which agrees with the form given in \eqref{E3.5.1} when restricted to 
the subalgebra $B=A^{cov\; \Bbbk^G}$.

\begin{proposition}
\label{xxpro4.5} 
Suppose $G$ coacts on $A$ as a dual reflection group.
\begin{enumerate}
\item[(1)]
$f_m$ is a normal element.
\item[(2)]
The Nakayama automorphism of $A^{cov\; \Bbbk^G}$
is equal to the induced automorphism of $\phi_{m}$
of $A$ given in \eqref{E4.4.3}.
\end{enumerate}
\end{proposition}

\begin{proof} (1) Let $B=A^{cov\; \Bbbk^G}$. If $b\in B$, then 
$f_m b= b' f_m$ by Theorem \ref{xxthm3.5}(1). If $x=f_g$, 
we have
$$\begin{aligned}
f_{m} f_{g} &=\alpha(mgm^{-1},mg^{-1})^{-1} f_{mgm^{-1}}
f_{mg^{-1}}f_g \\
&=\alpha(mgm^{-1},mg^{-1})^{-1} f_{mgm^{-1}}
\alpha(mg^{-1},g)f_{m}\\
&= \beta(g) f_{mgm^{-1}} f_{m}.
\end{aligned}
$$
Since $A$ is generated by $B$ and $\{f_g\mid g\in G\}$,
$f_m$ is normal.

(2) By the computation in the proof of part (1),
$$f_{m} f_{g}= \beta(g) f_{mgm^{-1}} f_{m}$$
which is equivalent to 
$\phi_{m}(f_g)=\beta(g) f_{mgm^{-1}}$ for all $g\in G$.
The assertion follows by \eqref{E4.4.1}.
\end{proof}

We finish this section with proofs of the main results.

\begin{proof}[Proof of Theorem \ref{xxthm0.3}]
(1) This is Theorem \ref{xxthm3.5}(1).

(2) This follows by combining parts (3) and (4) 
of Theorem \ref{xxthm3.5}.

(3) By Theorem \ref{xxthm3.5}(5), $A^{cov\; \Bbbk^G}$
is a skew Hasse algebra. Note that $\deg (f_g)=
\deg (\overline{f_g})$ where $\overline{f_g}$ is
the image of $f_g$ in $A^{cov\; \Bbbk^G}$. As a $G$-graded
vector space, any skew Hasse algebra is isomorphic to 
the associated Hasse algebra. The assertion follows
by Remark \ref{xxrem2.4}(3). 

(4) We use $m$ to denote the element in $G$ of the maximal
length with respect to $\lrr$. By Remark \ref{xxrem4.3}(1),
$m$ equals the mass element defined in Definition \ref{xxdef0.2}.

(5) 
Since $A$ is free over $A^H$ and $A\cong A^H\otimes A^{cov\; H}$
as graded vector spaces, $\fp(t)$ is the Hilbert series of
$A^{cov\; H}$. Since $A^{cov\; \Bbbk^G}$ is a skew Hasse algebra,
$\fp(1)=|G|$ and $\deg \fp(t)=\lrr(m)$. By Theorem \ref{xxthm3.5}(5),
$\fp(t)$ is palindrome and is a product of cyclotomic
polynomials.
\end{proof}

\begin{proof}[Proof of Theorem \ref{xxthm0.4}]
(1) By Lemma \ref{xxlem3.4}(2), $C=A^{cov\; \Bbbk^G}$ is Frobenius.
Since $A^{cov\; \Bbbk^G}$ is a skew Hasse algebra 
[Theorem \ref{xxthm3.5}(5)], $\dim A^{cov\; \Bbbk^G}=|G|$.

(2) The first assertion is proven in the proof of Theorem \ref{xxthm0.3}(5).
The consequence is clear.

(3) This is \eqref{E4.4.1} or the proof of Proposition \ref{xxpro4.5}.

(4) By Proposition \ref{xxpro4.5}(2), $\mu_{A^{cov\; \Bbbk^G}}$ 
is induced by the conjugation $\phi_{m}: x\to f_{m} x f_{m}^{-1}$.
Note that $\phi_{m}(f_m)=f_{m}$, which implies that 
$\beta(m)=1$. For any $g,h\in G$ with $\lrr(gh)=\lrr(g)+\lrr(h)$,
we have $f_g f_{h}=\alpha(g,h) f_{gh}$ for some nonzero
scalar $\alpha(g,h)$. Then 
$$\begin{aligned}
\phi_{m}(f_{gh})&= \alpha(g,h)^{-1} \phi_{m}(f_g f_h)\\
&=\alpha(g,h)^{-1}\phi_{m}(f_g) \phi_{m}(f_h)\\
&=\alpha(g,h)^{-1}\beta(g) \beta(h) f_{mgm^{-1}} f_{mhm^{-1}}\\
&= \alpha(g,h)^{-1}\beta(g) \beta(h) \alpha(mgm^{-1}, mhm^{-1}) 
f_{mghm^{-1}}\\
&= \alpha(g,h)^{-1}\alpha(mgm^{-1}, mhm^{-1}) 
\beta(g) \beta(h) f_{mghm^{-1}}.
\end{aligned}
$$
Hence $\beta(gh)=\alpha(g,h)^{-1}\alpha(mgm^{-1}, mhm^{-1}) 
\beta(g) \beta(h)$ for all $g,h$ satisfying  $\lrr(gh)=\lrr(g)+\lrr(h)$.
If $m$ commutes with $g$ and $h$, then $\beta(gh)=\beta(g) \beta(h)$.
\end{proof}

\begin{proof}[Proof of Theorem \ref{xxthm0.5}]
(1) This is Proposition \ref{xxpro4.5}(1).

(2) See Proposition \ref{xxpro4.5}(1) and its proof.

(3) This follows from \cite[Lemma 5.3(b)]{RRZ2}.

(4) This follows from Theorem \ref{xxthm3.5}(1).
\end{proof}

\begin{proof}[Proof of Theorem \ref{xxthm0.6}]
(1) This is Lemma \ref{xxlem4.1}(4).

(2) By \cite[Lemma 5.3(b)]{RRZ2}, $\mu_{A^{\Bbbk^G}}$ is 
in the center of the group $\Aut_{gr}(A^{\Bbbk^G})$. The assertion follows
from part (1).

(3) This is Proposition \ref{xxpro4.2}(1).

(4) This is Corollary \ref{xxcor4.4}.
\end{proof}

\section{Rigidity}
\label{xxsec5}

In this section we prove that some families of AS regular 
algebras are rigid with respect to group coactions. 
Assume that $\Bbbk$ is an algebraic closed field of 
characteristic zero in this section. 

\begin{lemma}
\label{xxlem5.1} Let $A$ be a noetherian AS regular domain
generated in degree one and $G$ be a non-trivial finite
group coacting on $A$ homogeneously and inner faithfully 
as a dual reflection group. Then there is a finite set of
nonzero elements $\{z_1,\cdots, z_w\}$ in degree one such that 
the product $z_1 z_2\cdots z_w$ is a normal element in $A$.
\end{lemma}

\begin{proof} Let $m$ be the mass element as in Definition 
\ref{xxdef0.2} and let $l_{\Re}(m)=w$ be the length of $m$
with respect to $\Re$ as given in Theorem \ref{xxthm0.3}(2).
Write $m=g_1 g_2 \cdots g_w$ where $g_i\in \Re$. 

Let $f_{g}$ be the element in $A$ as defined in Theorem 
\ref{xxthm0.3}. Since $l_{\Re}(m)=\sum_{i=1}^w l_{\Re}(g_i)$,
by the discussion after Theorem \ref{xxthm3.5}, \eqref{E3.5.2}
and induction, one see that 
\begin{equation}
\label{E5.1.1}\tag{E5.1.1}
f_{m}=c f_{g_1}f_{g_2}\cdots f_{g_w}
\end{equation}
for some nonzero scalar $c$. Since $f_{m}$ is normal 
by Theorem \ref{xxthm0.5}(1), the assertion follows
by setting $z_i= f_{g_i}$ for $i=1,2,\cdots,w$.
\end{proof}

Let ${\mathfrak g}$ be a finite dimensional Lie algebra, 
and $U({\mathfrak g})$ be the universal enveloping algebra
of ${\mathfrak g}$. Let $H({\mathfrak g})$ be the homogenization
of $U({\mathfrak g})$. Note that $H({\mathfrak g})$
is a connected graded algebra generated by the vector space 
${\mathfrak g}\oplus  \Bbbk  t$ subject to the relations
$$at = ta, \quad {\text{and}} \quad ab-ba=[a,b]t$$
for all $a,b\in {\mathfrak g}$. 
It is well known that $H({\mathfrak g})$ is a noetherian 
AS domain of global dimension $d:=\dim {\mathfrak g}+1$ and
its Hilbert series is $(1-t)^{-d}$. 

The proof of part (2) of the following lemma is due to Monty 
McGovern. We thank him for allowing us to include his proof 
here. 

\begin{lemma}
\label{xxlem5.2} 
Let $H({\mathfrak g})$ be the homogenization of the universal
enveloping algebra of a finite dimensional semisimple Lie algebra.
\begin{enumerate}
\item[(1)]
If $f\in H({\mathfrak g})$ is a homogeneous normal element, then $f$ is central.
\item[(2)]
There is no central element $f\in U({\mathfrak g})$ such that $f$ 
is a nontrivial product of elements in $(\Bbbk+{\mathfrak g})\setminus \Bbbk$.
\end{enumerate}
\end{lemma}

\begin{proof} (1) Let $f=t^i f_0$ where $i\geq 0$ and $f_0$ does not have a 
factor of $t$. Since $t$ is central, we may assume that $f=f_0$ which 
is not divisible by $t$. For any $\ell\in {\mathfrak g}$, since $f$ is normal
and homogeneous, we have 
\begin{equation}
\label{E5.2.1}\tag{E5.2.1}
f(\ell)=(a(\ell) t +\ell') f
\end{equation}
for some $a(\ell)\in \Bbbk$. Passing equation \eqref{E5.2.1} to the quotient
ring $S({\mathfrak g})=H({\mathfrak g})/(t)$, we obtain that
$$\bar{f} \ell= \ell' \bar{f}$$
where $\bar{f}\neq 0$ as $f$ is not divisible by $t$. Since $S({\mathfrak g})$
is commutative, $\ell=\ell'$. Thus, for every $\ell\in {\mathfrak g}$,
$$[f,\ell]=f\ell-\ell f=a(\ell) tf$$
for some $a(\ell)\in \Bbbk$. It is easy to check that 
$$[f,[\ell_1,\ell_2]t]=(a(\ell_1)a(\ell_2)-a(\ell_2)a(\ell_1)) t^2 f=0$$
for all $\ell_1,\ell_2\in {\mathfrak g}$. Since ${\mathfrak g}$ is semisimple,
$[f, \ell]=0$ for all $\ell\in {\mathfrak g}$. Since $t$ is central,
$f$ commutes with all elements in degree 1. The assertion follows. 

(2)  Let $G$ be the Lie group associated to ${\mathfrak g}$ and 
consider the adjoint action of $G$ on ${\mathfrak g}$ that extends naturally 
to the action on both the symmetric algebra $S({\mathfrak g})$ and the 
enveloping algebra $U({\mathfrak g})$.  Given any product 
$$f:=(a_1+\ell_1) \cdots (a_n+\ell_n), \quad
a_i\in \Bbbk, 0\neq \ell_i\in {\mathfrak g},$$ 
of elements of $(\Bbbk+{\mathfrak g})\setminus \Bbbk$ in $U({\mathfrak g})$, for some $n\geq 1$,
assume to the contrary that $f$ is in the center of $U({\mathfrak g})$. By an 
elementary property of the adjoint representation, we have $g(f)=f$ for all 
$g\in G$. By using the standard filtration, $g(\bar{f})=\bar{f}$ for all $g\in G$ 
when $\bar{f}:=\ell_1\cdots \ell_n$ is considered as an element in $S({\mathfrak g})$. 
Now we choose $g$ in $G$ such that  $g(\ell_1)$ is not a scalar multiple of 
$\ell_i$ for any $i$. This is possible since ${\mathfrak g}$ is semisimple and 
$G$ acts on ${\mathfrak g}$ with no nonzero fixed points.  By unique factorization 
in $S({\mathfrak g})$, $g(\ell_1 \cdots\ell_n)$, which is 
$g(\ell_1)\cdots g(\ell_n)$, cannot coincide with $\ell_1\cdots\ell_n$.
Therefore $g(\bar{f})\neq \bar{f}$, yielding a contradiction.
\end{proof}

\begin{lemma}
\label{xxlem5.3} 
Let $A$ be the homogenization of the universal enveloping algebra of a finite 
dimensional semisimple Lie algebra $H({\mathfrak g})$. Let $\{z_1,\cdots, z_w\}
\subseteq A$ be a set of nonzero elements of degree one such that the product 
$z_1 z_2\cdots z_w$ is a normal element in $A$. Then each $z_i$ is a scalar 
multiple of $t$. 
\end{lemma}

\begin{proof} Since $t$ is central, we can remove those $z_i$ of the form
$at$ for some $a\in \Bbbk$. Thus each $z_i$ is $a_it+\ell_i$ where $a_i\in \Bbbk$ 
and $0\neq \ell_i\in {\mathfrak g}$. By Lemma \ref{xxlem5.2}(1), 
$z_1\cdots z_n$ is central. Then $f:=\pi(z_1)\cdots \pi(z_n)$ 
is central in $U({\mathfrak g})$ where $\pi$ is the quotient map 
$H({\mathfrak g})\to H({\mathfrak g})/(t-1)=U({\mathfrak g})$. 
By Lemma \ref{xxlem5.2}(2), $f$ is trivial. So $n=0$. This means that
each $z_i$ is of the form $at$ for $a\in \Bbbk$.
\end{proof}

Let $A_n(\Bbbk)$ be the $n$th Weyl algebra generated by
$x_1,\cdots,x_n, y_1,\cdots,y_n$ subject to the relations
$$[x_i,x_j]=0=[y_i,y_j], \quad [x_i,y_j]=\delta_{ij}.$$
The Rees ring of $A_n(\Bbbk)$ with respect to the 
standard filtration is generated by 
$x_1,\cdots,x_n, y_1,\cdots,y_n, t$ subject to the relations
$$[x_i,x_j]=0=[y_i,y_j]=[t,x_i]=[t,y_i], \quad [x_i,y_j]=\delta_{ij}t^2.$$

\begin{lemma}
\label{xxlem5.4} 
Let $A$ be the Rees ring of the Weyl algebra $A_n(\Bbbk)$ with respect to 
the standard filtration. Let $\{z_1,\cdots, z_w\}\subseteq A$ be a set of 
nonzero elements of degree one such that the product $z_1 z_2\cdots z_w$ 
is a normal element in $A$. Then each $z_i$ is a scalar multiple of $t$. 
\end{lemma}

\begin{proof} Up to a scalar $z_i$ is of the form $t+f_i$ or $f_i$ 
where $f_i\in V:=\bigoplus_{i=1}^n (\Bbbk x_i+\Bbbk y_i)$. When 
$z_i=f_i$, then $f_i\neq 0$. Let $z=z_1 z_2\cdots z_w$ and consider the
algebra map $\phi: A\to A/(t-1)=A_n(\Bbbk)$. Then $\phi(z_i)$ is either
$1+f_i$ or $f_i$, which is nonzero. Since $z$ is normal, so 
is $\phi(z)$. But $A_n(\Bbbk)$ is simple, which implies that 
$\phi(z)$ is a scalar. In this case, the only possibility 
is $z_i=t$ for all $i$. 
\end{proof}

\begin{lemma}
\label{xxlem5.5} 
Let $A$ be the 
non-PI Sklyanin algebra of global dimension at least 3.
Then there is no finite set of
nonzero elements $\{z_1,\cdots, z_w\}$ in degree one such that 
the product $z_1 z_2\cdots z_w$ is a normal element in $A$.
\end{lemma}

\begin{proof} This was basically proved in \cite[Corollary 6.]{KKZ1}.
For completeness we give a proof here. Let $n$ be the global 
dimension of the Sklyanin algebra $A$. 

Associated to $A$ there is a triple $(E, \sigma, {\mathcal L})$ 
where $E \subseteq {\mathbb P}^{n-1}$ is an elliptic curve of
degree $n$, ${\mathcal  L}$ is an invertible line bundle over 
$E$ of degree $n$ and $\sigma$ is an automorphism of $E$ induced 
by a translation. Basic properties of $A$ can be found in 
\cite{ATV} for $n = 3$, \cite{SS} for $n = 4$, and \cite{TV} 
for $n \geq 5$. Associated to $(E, \sigma,{\mathcal L})$ one can
construct the twisted homogeneous coordinate ring, denoted by 
$B(E, \sigma,{\mathcal L})$. Then there is a canonical surjection
$$\phi: A \to  B(E, \sigma,{\mathcal L}) =: B$$
such that $\phi$ becomes an isomorphism when restricted to the degree 
one piece. This statement was proved in 
\cite[Section 6]{ATV} for $n = 3$, \cite[Lemma 3.3]{SS} for $n = 4$,
and \cite[(4.3)]{TV} for $n \geq 5$. If $A$ is non-PI, then $\sigma$ 
has infinite order. Hence $B$ is so-called {\it projectively simple} 
by \cite{RRZ1}, which means that any proper factor ring of B is finite
dimensional. Also note that the GK-dimension of $B$ is 2.

Suppose that there are nonzero elements $z_1,\cdots,z_w$ in $A$ of 
degree 1, such that $x:=z_1 z_2\cdots z_w$ is normal. Let 
$$x'=\phi(x)=\phi(z_1)\cdots \phi(z_w)\in B.$$ 
Since $\phi$ is an isomorphism in degree 1, each $\phi(z_i)\neq 0$. 
Now a basic property of $B$ is that it is a domain. Hence 
$x'\neq 0$. Since $x$ is normal, so is $x'$. Therefore $B/(x')$ is
an infinite dimensional proper factor ring of $B$, which contradicts 
the fact that $B$ is projectively simple.
\end{proof}

%\begin{lemma}
%\label{xxlem5.5} 
%Let $A$ be the 
%noetherian graded Down-up algebras.
%Then there is no finite set of
%nonzero elements $\{z_1,\cdots, z_w\}$ in degree one such that 
%the product $z_1 z_2\cdots z_w$ is a normal element in $A$.
%\end{lemma}

Now we are ready to prove Theorem \ref{xxthm0.9}.

\begin{theorem}
\label{xxthm5.6}
Let $\Bbbk$ be an algebraically closed field of characteristic zero. 
\begin{enumerate}
\item[(1)]
Let $A$ be the homogenization of the universal
enveloping algebra of a finite dimensional semisimple Lie algebra 
$H({\mathfrak g})$. For every finite group $G$, $A^{\Bbbk^G}$ 
is not AS regular. As a consequence, $A$ is rigid with respect to 
group coactions.
\item[(2)]
Let $A$ be the Rees ring of the Weyl algebra $A_n(\Bbbk)$ with respect 
to the standard filtration. If $G$ is a finite group such that
$A^{\Bbbk^G}$ is AS regular, then $G={\mathbb Z}/(2)$ and 
$A^{\Bbbk^G}\not\cong A$. As a consequence, $A$ is rigid with respect to 
group coactions.
\item[(3)]
Let $A$ be the non-PI Sklyanin algebras of global dimension at least 3.
For every finite group $G$, $A^{\Bbbk^G}$ 
is not AS regular. As a consequence, $A$ is rigid with respect to 
group coactions.
\end{enumerate}
\end{theorem}

\begin{proof}
(1) Assume to the contrary that $G$ is a nontrivial finite group
such that $A^{\Bbbk^G}$ is AS regular. By Lemma \ref{xxlem5.1} and 
\eqref{E5.1.1}, $f_m=cf_{g_1}\cdots f_{g_w}$ where 
$g_i\in \Re$. Then $f_{g_i}$ is of degree 1. By 
Theorem \ref{xxthm0.5}(1), $f_m$ is a normal 
element. By Lemma \ref{xxlem5.3}, each $f_{g_i}$ is a
scalar multiple of $t$. This implies that all
$g_i$ are the same and $G$ is generated by $g_1$,
whence abelian. In this case the Hopf algebra
$\Bbbk^G$ is cocommutative. Since $\Bbbk$ is 
algebraically closed, $H$ is isomorphic to 
a group algebra $\Bbbk G_0$ for some group
$G_0$ (in fact $G_0\cong G$ is cyclic). By \cite[Theorem 2.4]{KKZ1},
$G_0$ contains a quasi-reflection. This contradicts  
\cite[Lemma 6.5(d)]{KKZ1}. Therefore the assertion follows,
and the consequence is obvious. 

(2) Let $G$ be a nontrivial finite group such that $A^{\Bbbk^G}$ 
is AS regular. By the proof of part (1), there is a 
nontrivial cyclic group $G_0$ such that $A^{\Bbbk^G}=A^{G_0}$
with some natural action of $G_0$ on $A$ and $\Bbbk^G\cong \Bbbk G_0$. 
By \cite[Proposition 6.7]{KKZ1}, $G_0={\mathbb Z}/(2)$ and
by \cite[Corollary 6.8]{KKZ1}, $A^{G_0}\not\cong A$. Note that $G\cong G_0
={\mathbb Z}/(2)$. So the assertion follows, and the consequence is obvious.

(3) The assertion follow from Lemmas \ref{xxlem5.1} and \ref{xxlem5.5}.
The consequence is obvious.
\end{proof}

\subsection*{Acknowledgments}

The authors would like to thank Monty McGovern for his proof of Lemma \ref{xxlem5.2}(2).
E. Kirkman was partially supported by
grant \#208314 from the Simons Foundation. J.J. Zhang was partially supported
by the US National Science Foundation (grant No DMS 1402863).

\end{document}